\documentclass[a4paper,leqno,12pt]{amsart}
\usepackage[all]{xy}
\usepackage{xspace}
\usepackage{amsmath}
\usepackage{amstext}
\usepackage{amsfonts}
\usepackage[mathscr]{euscript}
\usepackage{amscd}
\usepackage{latexsym}
\usepackage{amssymb}
\usepackage{diagrams}
\theoremstyle{plain}
\swapnumbers
    \newtheorem{thm}{Theorem}[section]
    \newtheorem{prop}[thm]{Proposition}
    \newtheorem{lemma}[thm]{Lemma}
    \newtheorem{corollary}[thm]{Corollary}
    \newtheorem{fact}[thm]{Fact}
    \newtheorem{subsec}[thm]{}
\theoremstyle{definition}
    \newtheorem{defn}[thm]{Definition}

\theoremstyle{remark}
        \newtheorem{remark}[thm]{Remark}
        
    \newtheorem{assume}[thm]{Assumption}
    \newtheorem{ack}[thm]{Acknowledgements}
\newenvironment{myeq}[1][]
{\stepcounter{thm}\begin{equation}\tag{\thethm}{#1}}
{\end{equation}}

\newcommand{\mydiagram}[2][]
{\stepcounter{thm}\begin{equation}
     \tag{\thethm}{#1}\vcenter{\xymatrix{#2}}\end{equation}}
\newenvironment{mysubsection}[2][]
{\begin{subsec}\begin{upshape}\begin{bfseries}{#2.}
\end{bfseries}{#1}}
{\end{upshape}\end{subsec}}
\newcommand{\sect}{\setcounter{thm}{0}\section}
\newcommand{\wh}{\ -- \ }
\newcommand{\wwh}{-- \ }
\newcommand{\w}[2][ ]{\ \ensuremath{#2}{#1}\ }
\newcommand{\ww}[1]{\ \ensuremath{#1}}
\newcommand{\wwb}[1]{\ \ensuremath{(#1)}-}
\newcommand{\wb}[2][ ]{\ (\ensuremath{#2}){#1}\ }
\newcommand{\wref}[2][ ]{\ (\ref{#2}){#1}\ }
\newcommand{\hsp}{\hspace*{7 mm}}

\newcommand{\hsm}{\hspace*{2 mm}}
\newcommand{\vsn}{\vspace{1 mm}}

\newcommand{\vsm}{\vspace{3 mm}}
\newcommand{\hra}{\hookrightarrow}
\newcommand{\xra}[1]{\xrightarrow{#1}}
\newcommand{\xepic}[1]{\xrightarrow{#1}\hspace{-5 mm}\to}
\newcommand{\lora}{\longrightarrow}
\newcommand{\lra}[1]{\langle{#1}\rangle}
\newcommand{\Rw}{\Rightarrow}

\newcommand{\rest}[1]{\lvert_{#1}}
\newcommand{\Aut}{\operatorname{Aut}}
\newcommand{\colim}{\operatorname{colim}}

\newcommand{\he}{\operatorname{h.e.}}

\newcommand{\Hom}{\operatorname{Hom}}
\newcommand{\Id}{\operatorname{Id}}
\newcommand{\lenG}{\operatorname{len}_{G}}
\newcommand{\Obj}{\operatorname{Obj}\,}
\newcommand{\op}{\sp{\operatorname{op}}}

\newcommand{\pt}{\operatorname{pt}}
\newcommand{\rst}{\operatorname{rest}}
\newcommand{\sk}{\operatorname{sk}}
\newcommand{\map}{\operatorname{Map}}

\newcommand{\mapg}{\map\sb{G}}
\newcommand{\hy}[2]{{#1}\text{-}{#2}}
\newcommand{\Abgp}{{\EuScript AbGp}}
\newcommand{\C}{{\EuScript C}}
\newcommand{\eE}{{\EuScript E}}
\newcommand{\EH}[1]{\eE\sp{H}\sb{#1}}
\newcommand{\tEH}[1]{\widetilde{\eE}\sp{H}\sb{#1}}

\newcommand{\cO}{{\EuScript O}}
\newcommand{\eS}{{\EuScript S}}

\newcommand{\Sa}{\eS_{\ast}}

\newcommand{\Top}{{\EuScript T}}

\newcommand{\GT}{\hy{G}{\Top}}

\newcommand{\uC}{\underline{C}}
\newcommand{\uD}{\underline{D}}
\newcommand{\uf}{\underline{f}}

\newcommand{\uh}{\underline{h}}
\newcommand{\uth}[1]{\widetilde{\uh\sp{#1}}}
\newcommand{\uK}{\underline{K}}
\newcommand{\uM}{\underline{M}}
\newcommand{\uS}{\underline{S}}
\newcommand{\us}{\underline{s}}
\newcommand{\uW}{\underline{W}}
\newcommand{\uX}{\underline{X}}

\newcommand{\uY}{\underline{Y}}
\newcommand{\uZ}{\underline{Z}}
\newcommand{\BG}{\bB\Gamma}
\newcommand{\EG}{\bE\Gamma}
\newcommand{\HB}[4]{\bH^{#1}_{#2}({#3};\,{#4})}
\newcommand{\HL}[4]{H^{#1}_{#2}({#3};\,{#4})}
\newcommand{\tH}[4]{\tilde{H}^{#1}_{#2}({#3};\,{#4})}
\newcommand{\KM}[1]{\bK(M,{#1})}
\newcommand{\KGM}[1]{\bK_{\Gamma}(M,{#1})}
\newcommand{\MH}{M_{H}}

\newcommand{\ML}{M_{L}}
\newcommand{\EX}[2]{\bE{#1}\times{#2}}
\newcommand{\EXp}[2]{\bE{#1}\ltimes{#2}}
\newcommand{\Br}[2]{{#1}_{h{#2}}}
\newcommand{\Bor}[2]{\bE{#2}\times_{#2}{#1}}
\newcommand{\Brp}[2]{{#1}^{\ast}_{h{#2}}}
\newcommand{\Borp}[2]{\bE{#2}\ltimes_{#2}{#1}}
\newcommand{\XWH}{\bE\WH\times_{\WH}\XHH}

\newcommand{\XWL}{\bE\WL\times_{\WL}\XLL}
\newcommand{\XWLz}{\bE W_{L_{0}}\times_{W_{L_{0}}}\XLLz}
\newcommand{\WH}{W_{H}}

\newcommand{\WK}{W_{K}}
\newcommand{\WL}{W_{L}}

\newcommand{\bWHK}{\bar{W}^{H}_{K}}

\newcommand{\bWHL}{\bar{W}^{H}_{L}}

\newcommand{\WKH}{W^{K}_{H}}
\newcommand{\WLH}{W^{L}_{H}}
\newcommand{\WHLp}{W^{H'}_{L}}
\newcommand{\WLHp}{W^{L'}_{H}}
\newcommand{\XHH}{\bX^{H}_{H}}
\newcommand{\XHHp}{\bX^{H'}_{H'}}
\newcommand{\XLL}{\bX^{L}_{L}}
\newcommand{\XLLz}{\bX^{L_{0}}_{L_{0}}}
\newcommand{\XLLp}{\bX^{L'}_{L'}}
\newcommand{\F}{\mathcal{F}}
\newcommand{\hF}{\widehat{\F}}

\newcommand{\OG}{\cO_{G}}
\newcommand{\OWH}{\cO_{\WH}}
\newcommand{\cP}{\mathcal{P}}
\newcommand{\TE}[1]{\Top^{\tEH{#1}}}
\newcommand{\TF}[1]{\Top\sp{\F\sb{#1}}}
\newcommand{\TOG}{\Top\sp{\OG\op}}
\newcommand{\Q}[1]{\mathcal{Q}\sb{#1}}
\newcommand{\cS}{\mathcal{S}}
\newcommand{\hS}{\widehat{\cS}}
\newcommand{\tSH}[1]{\widetilde{\cS}\sp{H}\sb{#1}}
\newcommand{\cV}{\mathcal{V}}
\newcommand{\LH}[1]{\Lambda\sp{H}\sb{#1}}

\newcommand{\cj}[2]{\phi\sp{#1}\sb{#2}}
\newcommand{\co}[2]{(\phi\sp{#1}\sb{#2})\op}
\newcommand{\cp}[2]{(\widetilde{\phi}\sp{#1}\sb{#2})\op}
\newcommand{\hf}{\widehat{f}}

\newcommand{\hD}{\widehat{D}}
\newcommand{\hM}{\widehat{M}_{H}}
\newcommand{\hX}{\widehat{X}}
\newcommand{\hY}{\widehat{Y}}
\newcommand{\hZ}{\widehat{Z}}
\newcommand{\tU}[1]{\widetilde{U}\sb{#1}}
\newcommand{\bB}{\mathbf{B}}
\newcommand{\bD}[1]{\mathbf{D}\sp{#1}}
\newcommand{\bE}{\mathbf{E}}
\newcommand{\bH}{\mathbf{H}}
\newcommand{\bK}{\mathbf{K}}
\newcommand{\bS}[1]{\mathbf{S}\sp{#1}}
\newcommand{\bX}{\mathbf{X}}
\newcommand{\bY}{\mathbf{Y}}
\newcommand{\bZ}{\mathbf{Z}}
\begin{document}
%
%           Title
%
\title{A spectral sequence for Bredon cohomology}
\author{David Blanc}
\author{Debasis Sen}
\address{Department of Mathematics\\ University of Haifa\\ 31905 Haifa\\ Israel}
\email{blanc@math.haifa.ac.il,\ sen\_deba@math.haifa.ac.il}

\date{\today}


\begin{abstract}
For any finite group $G$, we construct a spectral sequence for
computing the Bredon cohomology of a $G$-CW complex $\bX$, starting with
the cohomology of \w{\bX^{H}/\bigcup_{K>H}\bX^{K}} with suitable local coefficients, for various \w[.]{H\leq G}
\end{abstract}

\maketitle

\setcounter{section}{0}

%
%c0   Introduction
%
\section*{Introduction}
\label{cint}

Equivariant homotopy theory is the study of \emph{$G$-spaces} \wh topological spaces equipped with a continuous action of a group $G$ \wh using homotopy-theoretic methods. In \cite{BredE}, Bredon proposed a framework for studying a $G$-space $\bX$ using the system of fixed point sets \w{\bX^{H}} for
various subgroups \w[.]{H\leq G} In particular, he introduced an
equivariant cohomology theory \w[,]{\HB{\ast}{G}{\bX}{\uM}} for any
coefficient system \w{\uM:\OG\op\to\Abgp} defined on the orbit category
\w{\OG} (cf.\ \S \ref{socat}).

Bredon cohomology has become one of the major theoretical tools of
equivariant homotopy theory. However, it is notoriously
difficult to calculate. Our goal here is to describe a spectral sequence
converging to \w[,]{\HB{\ast}{G}{\bX}{\uM}} for a finite group $G$,
starting from ``local'' information at the various fixed point sets
\w[.]{\bX^{H}}
In fact, the spectral sequence takes the form
$$
E_{1}^{k,i}~=~\bigoplus_{[G/H]}
\tH{i}{\WH}{\XWH}{\MH}~~\Longrightarrow~~\HB{i}{G}{\bX}{\uM}~,
$$
\noindent where \w{\tH{\ast}{\Gamma}{-}{M}} are reduced local
cohomology groups, \w[,]{\WH:=N_{G}H/H} and
\w[.]{\XHH=\bX^{H}/\bigcup_{K>H}\bX^{K}} See Theorem \ref{tss} below.

The idea for this spectral sequence is based on a more general
construction of local-to-global spectral sequences for the cohomology
of a diagram \w{\uX:I\to\C} (cf.\ \cite{BJTurL}, and compare
\cite{JPirCA,MRobiC}); however, the latter only works for
\emph{directed} indexing categories $I$, so it does not apply to
Bredon cohomology. Note also that Moerdijk and Svensson have a
different construction of a spectral sequence for computing Bredon
cohomology (see \cite{MSvenS}).

One might expect the spectral sequence constructed here to start from
the Bredon cohomology \w{\HB{\ast}{\WH}{\bX^{H}}{\hM}} at the various
fixed point sets. We were not able to obtain such a spectral sequence
directly. However, we do have another spectral sequence of the form:
$$
E_{1}^{m,i}~=~\bigoplus_{G/L}
\tH{i}{\bWHL}{\bE\WL\times_{\bWHL}\XLL}{M_{N_{G}H\cap L}}~~
\Longrightarrow~~\HB{i}{\WH}{\bX^{H}}{\hM}~,
$$
\noindent thus allowing us to compute these fixed-point-set Bredon
cohomology groups, too, from the reduced cohomology with local coefficients.
See Theorem \ref{tssfps}.

\begin{remark}\label{rfinite}
There is a version of Bredon cohomology for any topological group $G$ (in particular, for a compact Lie group), based on choosing a family of (closed) subgroups of $G$ (see \cite{DKanEq,IllmE1}). However, in this case the orbit category \w{\OG} is itself topologically (or simplicially) enriched, so the
diagram systems involved are more complicated. Our methods do not
work in this situation, in general, because the filtration
\wref{eqfilter} that we use need not be exhaustive. Therefore,
throughout this paper we assume that $G$ is a \emph{finite} group.

We would like to thank Wolfgang L\"{u}ck for pointing out that our spectral sequence is similar in spirit to the much more general $p$-chain spectral sequence of \cite{DLuckP}.  However, our construction is different, and in this special case we can describe the \ww{E_{1}}-term and differential quite explicitly.
\end{remark}

\begin{mysubsection}{Notation and conventions}\label{snac}
All mapping spaces \w{\map(-,-)} are simplicial sets, and the
category of simplicial sets will be denoted by $\eS$.
The category of topological spaces will be denoted by $\Top$, and its
objects will be denoted by boldface letters: $\bX$, \w[.]{\bY\dotsc}
We use \w{\bH^{\ast}_{G}} for Bredon cohomology, to distinguish it from
cohomology with local coefficients, denoted simply by \w[.]{H^{\ast}_{\Gamma}}
\end{mysubsection}

\begin{mysubsection}{Organization}\label{sorg}
In Section \ref{cbc} we provide some background on $G$-spaces, the
orbit category \w[,]{\OG} and Bredon cohomology. In Section \ref{cfoc}
we define the filtration on the orbit category which is the basis for
our spectral sequences. In Section \ref{clocal} we recall some basic
facts about cohomology with local coefficients in our connection,
and in Section \ref{csss} we construct the two spectral sequences.
\end{mysubsection}

\begin{ack}
We would like to thank the referee for his or her comments and suggestions.
\end{ack}

%
%c1   Bredon cohomology}
%
\sect{Bredon cohomology}
\label{cbc}

Bredon introduced a cohomology theory for $G$-spaces, using the
following notions:

\begin{mysubsection}{The orbit category}\label{socat}
Let $G$ be a fixed (finite) group. A \emph{basic $G$-set} is the set of
left cosets \w{G/H} for some subgroup \w[,]{H\leq G} with the left $G$-action.
The \emph{orbit category} \w{\OG} of $G$ has the basic $G$-sets as
objects, and $G$-equivariant maps as morphisms.

Any map \w{G/H\to G/K} in \w{\OG} can be factored as an
epimorphism \w{G/H\xepic{i_{\ast}} G/K^{a^{-1}}} (induced by the inclusion
\w[),]{i:H\hra K^{a^{-1}}} followed by an isomorphism
\w[.]{\cj{K^{a^{-1}}}{a}:G/K^{a^{-1}}\to G/K} Here
\w[,]{K^{a^{-1}}=aKa^{-1}, a\in G} and \w{\cj{K^{a^{-1}}}{a}} is
induced by the right translation \w{R_{a}:G\to G} (with
\w[),]{R_{a}(g)=ag} that is:
\begin{myeq}[\label{eqrightmult}]
\cj{K^{a^{-1}}}{a}~:~g K^{a^{-1}}~=~gaKa^{-1} ~
\longmapsto~ gaK
\end{myeq}
\noindent (see \cite[I, \S 3]{BredE}). This map can also be decomposed as \w[,]{i^{a}_{\ast}\circ\cj{H}{a}} where \w{\cj{H}{a}:G/H\to G/H^{a}} is induced by
\w{R_{a}} and \w{i^a:H^{a}\hra K} is the conjugate of \w{i:H\hra K^{a^{-1}}} by
$a$.

\begin{fact}\label{fcommute}
Two maps \w{\cj{K^{a^{-1}}}{a}\circ i_{\ast}}
and \w{\cj{K^{b^{-1}}}{b}\circ j_{\ast}} from \w{G/H} to \w{G/K}
are the same in \w{\OG} if and only if \w{a^{-1}b\in K} (so
\w{K^{a^{-1}}=K^{b^{-1}}} and \w{\cj{K^{b^{-1}}}{ba^{-1}}} is the
  identity). Therefore, the automorphism group \w{\WH:=\Aut_{\OG}(G/H)} of \w{G/H\in\OG} is \w[,]{N_{G}H/H} where \w{N_{G}H} is the normalizer of $H$ in $G$.
\end{fact}

Note that if $a$ is in \w[,]{K^{a^{-1}}} the right multiplication
\w{R_{a}} induces the \emph{identity} map
\w[,]{\cj{K^{a^{-1}}}{a}:G/K^{a^{-1}}\to G/K} even though the conjugation isomorphisms \w{\rho^{H}_{a}:H\to H^{a}} and \w{\rho^{K^{a^{-1}}}_{a}:K^{a^{-1}}\to K} may be non-trivial.
\end{mysubsection}

\begin{mysubsection}{$G$-spaces}\label{sgspaces}
For any (finite) group $G$, a \emph{$G$-space} is a topological space
\w{\bX\in\Top} equipped with a left $G$-action. The category of
$G$-spaces with $G$-equivariant continuous maps, simply called
\emph{$G$-maps}, will be denoted by \w[.]{\GT} We write \w{\bX^{H}}
for the \emph{fixed point set}
\w{\{x\in\bX~:hx=x \ \forall h\in H\}} of $X$ under a subgroup
\w[.]{H\leq G}

The notion of $G$-CW complexes (for topological groups) was introduced in \cite{MatuG,IllmE1}: a $G$-\emph{CW complex} $\bX$ is the union of sub
$G$-spaces \w{\bX^{n}} such that
\w{\bX^{0}} is a disjoint union of basic $G$-sets \w[,]{G/H} and
\w{\bX^{n+1}} is obtained from \w{\bX^{n}} by attaching \emph{$G$-cells}
of the  form \w{G/H\times\bD{n+1}} (where \w{\bD{n+1}} is a
\wwb{n+1}disc with boundary \w[)]{\bS{n}} with attaching $G$-maps
\w[.]{G/H\times\bS{n}\to\bX^{n}} For finite $G$, this is equivalent to
$\bX$ being a CW-complex on which $G$ acts cellularly (see
\cite{DiecTG}). Subcomplexes and relative $G$-CW complexes
are defined in the obvious way. For any $G$-space $\bX$, there is a
$G$-CW complex \w{\widehat\bX} and a weak $G$-homotopy equivalence
\w[.]{\gamma:\widehat\bX\to\bX}

For any collection $\F$ of subgroups of $G$ closed under
conjugation, there is a simplicial model category structure on
\w[,]{\GT} due to Dwyer and Kan, in which
\begin{enumerate}
\renewcommand{\labelenumi}{(\roman{enumi})\ }
\item A $G$-map \w{f:\bX\to\bY} is a \emph{weak equivalence}
  (respectively, a \emph{fibration}) if for each \w[,]{H\in\F}
  the restriction \w{f\rest{\bX^{H}}}  is a weak equivalence
  (respectively, a Serre fibration).
\item A $G$-map \w{f:\bX\to\bY} is a \emph{cofibration} if it is a
  retract of a (transfinite) composite of inclusions of relative
  $G$-CW pairs (see \cite[\S 2.1(Q1)]{DKanSR}).
\item The function complex \w{\mapg(\bX,\bY)} for the
  simplicial structure on \w{\GT} is defined by:
\w{\mapg(\bX,\bY)_{n}:=\Hom_{\GT}(\bX\times\Delta[n],\bY)} (with a trivial
  $G$-action on \w[).]{\Delta[n]}
\end{enumerate}
\noindent See \cite[\S 2]{DKanSR}.
\end{mysubsection}

\begin{mysubsection}{$\OG$-diagrams}\label{sgdiagrams}
Bredon's approach to $G$-equivariant homotopy theory, extended by
Elmendorf in \cite{ElmS} to compact Lie groups, reduces the study of
a $G$-space $\bX$ to the system of fixed point sets under the subgroups
of $G$:

For any category $\C$, an \ww{\OG\op}-\emph{diagram} in $\C$ is a
functor \w[,]{\uX:\OG\op\to\C} and the category of all such diagrams
will be denoted by \w[.]{\C^{\OG\op}}
When $\C$ is a simplicial model category (cf.\ \cite[II, \S 2]{QuiH}),
\w{\C^{\OG\op}} has a projective simplicial model category
structure in which a map \w{f:\uX\to\uY} of \ww{\OG\op}-diagrams is a
weak equivalence (respectively, a fibration) if for each \w[,]{H\leq G}
\w{f(G/H):\uX(G/H)\to\uY(G/H)} is a weak equivalence (respectively, a
fibration). The mapping spaces are defined using the simplicial
structure in $\C$ by \w{\map_{\C^{\OG\op}}(\uX,\uY)_{n}:=
\Hom_{\C^{\OG\op}}(\uX\otimes\Delta[n],\uY)_{n}}
(cf.\ \cite[\S 1.3]{DKanSR} and compare \cite{PiacH}).
\end{mysubsection}

\begin{mysubsection}{$\OG$-diagrams in $\Top$}\label{sgdiagramstop}
When \w[,]{\C=\Top} the fixed point set functor \w[,]{\Phi:\GT\to\TOG}
sending a $G$-space $\bX$ to the diagram \w{\Phi\bX:\OG\op\to\Top} defined:
\begin{myeq}[\label{eqsingdiag}]
(\Phi\bX)(G/H)~:=~\bX^{H}~,
\end{myeq}
\noindent has a left adjoint \w{\Psi:\TOG\to\GT} (see
\cite[Theorem 1]{ElmS}). We shall usually denote \w{\Phi\bX} by $\uX$.

In fact, this adjoint pair constitutes a simplicial Quillen
equivalence between \w{\GT} and \w[.]{\TOG} See
\cite[Theorem 3.1]{DKanSR} for \w{\C=\eS} (with
\w[).]{\uX(G/H)~:=~\map_{G}(G/H,\bX)} Using the
``singular-realization'' adjoint pair
\w[,]{\eS\rightleftharpoons\Top} one can translate the Dwyer-Kan
result to our context.

For any topological space $\bZ$, with trivial $G$-action, the
associated \emph{basic $G$-spaces} are those of the form
\w{G/K\times\bZ} \wb[.]{K\leq G} We denote the corresponding fixed
point diagrams \w{\Phi(G/K\times\bZ)} by \w[,]{\uZ_{G/K}\in\TOG} so
\begin{myeq}[\label{eqncell}]
\uZ_{G/K}(G/H)~:=~\coprod_{\psi\op:G/K\to G/H~\text{in}~\OG\op}~\bZ_{\psi}~,
\end{myeq}
\noindent where \w{\bZ_{\psi}} is a copy of $\bZ$,
and the structure map \w{\uZ_{G/K}(\phi\op)} sends \w{\bZ_{\psi}} by the
identity homeomorphism to \w[.]{\bZ_{\psi\circ\phi}}
In particular, \w{\bZ_{\Id}} is the copy of $\bZ$ indexed by
\w{\Id:G/K\to G/K} in \w[,]{\uZ_{G/K}(G/K)} and we have:
\end{mysubsection}

\begin{fact}\label{funique}
For any \w[,]{\uY\in\TOG} a map of \ww{\OG\op}-diagrams
\w{\uf:\uZ_{G/K}\to\uY} is uniquely determined by a map of spaces
\w[.]{f:\bZ_{\Id}=\bZ\to\uY(G/K)}
\end{fact}

\begin{proof}
The summand \w{\bZ_{\psi}} in \w{\uZ_{G/K}(G/H)} is sent
to \w{\uY(G/H)} by \w[.]{\uY(\psi)\circ f}
\end{proof}

\begin{mysubsection}{Cellular $\OG$-diagrams}\label{scellular}
In particular, an \emph{$n$-cell} in \w{\TOG} is a diagram of the form
\w[,]{\uD^{n}_{G/K}:=\Phi(\bD{n}\times G/K)} where \w{\bD{n}} is an
$n$-cell in $\Top$, and similarly for the $n$-sphere \w[.]{\uS^{n}_{G/K}}
A \emph{cellular complex} in \w{\TOG} is a diagram
\w{\uX=\colim_{p}\uX^{p}} constructed inductively by a process of
``attaching cells'': i.e.,
\begin{myeq}[\label{eqattach}]
\uX^{p}~=~\uX^{p-1}~\cup_{(f_{\alpha})_{\alpha}}~
\coprod_{\alpha\in I_{p}}~\uD^{n_{\alpha}}_{G/K_{\alpha}}
\end{myeq}
\noindent for some indexing set \w{I_{p}} and diagram maps
\w[.]{f_{\alpha}:\uS^{n_{\alpha}-1}_{G/K_{\alpha}}\to\uX^{p-1}}
There is also a notion of a \emph{relative} cellular complex, and the
cofibrations in the model category \w{\TOG} are retracts of
inclusions of a relative cellular complex. See
\cite[Theorem 2.2]{DKanSR}.

The notion of a cellular diagram can be defined more generally \wh see
\cite[\S 3]{PiacH}.
\end{mysubsection}

\begin{fact}[cf.\ \protect{\cite{ElmS}}]\label{fgcw}
For any $G$-space $\bX$ and \w[,]{\uX\in\TOG} \w{\Psi\Phi\bX} is
a $G$-CW complex, \w{\Phi\Psi\uX} is a cellular diagram, the unit
\w{\bX\to\Psi\Phi\bX} is a $G$-weak homotopy equivalence, and the counit
\w{\Phi\Psi\uX\to\uX} is a weak equivalence of diagrams.
\end{fact}

\begin{assume}\label{agcw}
From now on we assume that all our $G$-CW-complexes are of the form
\w{\bX=\Psi\Phi\bX'} for some $G$-space \w[.]{\bX'}
\end{assume}

\begin{mysubsection}{Bredon cohomology}\label{sbredon}
Let $G$ be a (finite) group, and let \w{\uM:\OG\op\to\Abgp} be an
\ww{\OG\op}-diagram in abelian groups, known as a \emph{coefficient
  system} for $G$. Bredon showed that for each \w[,]{n\geq 1} one can
construct a natural \ww{\OG\op}-diagram \w{\uK(\uM,n):\OG\op\to\Top} with
each \w{\uK(\uM,n)(G/H):=\bK(\uM(G/H),n)} an Eilenberg-Mac~Lane space
(cf.\ \cite[\S 6]{BredE}). Equivalently, applying the functor $\Psi$ to
\w{\uK(\uM,n)} yields a $G$-space \w{\bK(\uM,n)} with the property that
\w{\bK(\uM,n)^{H}} is an (ordinary) Eilenberg-Mac~Lane space of type
\w[.]{\bK(\uM(G/H),n)} In particular, for \w{H=\{e\}} we see that
\w{\bK(\uM,n)} is an ordinary \w[.]{\bK(\uM(G/\{e\}),n)}

By \cite[Theorem 3.1]{DKanSR}, the function complex
\w{\map_{G}(\bX,\bY)} of $G$-maps between two $G$-CW
complexes $\bX$ and $\bY$ is weakly equivalent to the mapping space
\w{\map_{\TOG}(\uX,\uY)} between the corresponding
\ww{\OG\op}-diagrams (cf.\ \wref[),]{eqsingdiag} at least if $\uX$ is
cofibrant and $\uY$ is fibrant. Thus the study of $G$-maps between
$G$-spaces (up to homotopy) is reduced to the study of mapping spaces
of diagrams.

The $n$-th \emph{Bredon cohomology group} of a $G$-space $\bX$ with
coefficients in $\uM$ may then be defined to
be
\begin{myeq}[\label{eqbredon}]
\HB{n}{G}{\bX}{\uM}~:=~\pi_{0}\map_{\TOG}(\uX,\uK(\uM,n))~,
\end{myeq}
\noindent where we assume that $\uX$ is cofibrant and \w{\uK(\uM,n)} is
fibrant in the model category \w[.]{\TOG} See \cite[(6.1)]{BredE}.
Since $\Psi$ induces a simplicial Quillen equivalence, we obtain a
natural isomorphism
\begin{myeq}[\label{eqelmendorf}]
\HB{n-i}{G}{\bX}{\uM}~\cong~\pi_{i}\map_{G}(\bX,\bK(\uM,n))
\end{myeq}
\noindent for \w{0\leq i\leq n} (cf.\ \cite[V, \S 4]{MayEH}).

See \cite{BrockS} for a definition of singular equivariant (co)homology (for finite groups), \cite{WillsE} for cellular equivariant homology, and \cite{MatuEC} for cohomology, when $G$ is an arbitrary topological group.
\end{mysubsection}

%
%c2   Filtering the orbit category
%
\sect{ A filtration on the orbit category}
\label{cfoc}

The spectral sequence we construct here is based on the following filtration of \w[:]{\OG\op}

\begin{mysubsection}{Filtering $\OG\op$}\label{sfilter}
For any subgroup $H$ of $G$, we define the \emph{length} of $H$ in
$G$, denoted by \w[,]{\lenG H} to be the maximal \w{0\leq k<\infty}
such that there exists a sequence of proper inclusions of subgroups:
\begin{myeq}[\label{eqlength}]
H=H_{0}<H_{1}<H_{2}<\dotsc<H_{k-1}<H_{k}=G~.
\end{myeq}
\noindent This induces a filtration
\begin{myeq}[\label{eqfilter}]
\F_{0}~\subset~\F_{1}~\subset~\dotsc\F_{k}\subset~\dotsc~\subset~\OG\op
\end{myeq}
\noindent by full subcategories, where
\w{\Obj\F_{k}:=\{G/H\in\OG\op~:\ \lenG H\leq k\}} (so
\w[).]{\Obj\F_{0}=\{G/G\}}

Since $G$ is finite,  the filtration is exhaustive:
if \w{\lenG\{e\}=N} \wwh that is, the longest possible sequence
\wref{eqlength} in $G$ has $N$ inclusions of proper subgroups \wh then
\w[.]{\F_{N}=\OG\op}

Let \w{\EH{m}} denote the full subcategory of the slice category
\w{\OG\op/(G/H)} whose objects are \w{G/K\to G/H} in \w{\OG\op} with
\w[,]{G/K\in\F_{m}} with the obvious commuting triangles as maps.
\end{mysubsection}

\begin{defn}\label{dstrata}
Denote by \w{\cS_{k}} the $k$-th \emph{stratum} of the filtration \wh
that is, the full subcategory of \w{\OG\op} whose objects are in
\w[.]{\F_{k}\setminus\F_{k-1}} Note that all the maps
in \w{\cS_{k}} are isomorphisms (by the description in \S
\ref{socat}), and conversely, any isomorphism in \w{\OG\op} are
contained in some \w[.]{\cS_{k}} Any other map in \w{\OG\op} strictly
increases filtration.

We let \w{\hF_{k}} denote the collection of subgroups \w{H<G} such
that \w[,]{G/H\in\F_{k}} and \w[.]{\hS_{k}:=\hF_{k}\setminus\hF_{k-1}}
\end{defn}

\begin{mysubsection}{The tower of mapping spaces}\label{sstrata}
If $\C$ is any simplicially enriched category with colimits, the inclusion
\w{\F_{k}\hra\OG\op} induces a simplicial functor
\w{\tau_{k}:\C^{\OG\op}\to\C^{\F_{k}}} defined
\w[.]{\tau_{k}\uX:=\uX\rest{\F_{k}}}
Similarly, the inclusion \w{J_{k}:\F_{k}\hra\F_{k+1}} induces a
simplicial functor  \w[,]{J_{k}^{\ast}:\C^{\F_{k+1}}\to\C^{\F_{k}}}
so for any two diagrams \w{\uX,\uY:\OG\op\to\C} we obtain a tower of
simplicial sets
\begin{myeq}[\label{eqtower}]
\map(\tau_{N}\uX,\,\tau_{N}\uY)~\xra{p_{N}}~\dotsc~
\map(\tau_{k+1}\uX,\,\tau_{k+1}\uY)~\xra{p_{k+1}}~
\map(\tau_{k}\uX,\,\tau_{k}\uY)~\to~\dotsc~
\end{myeq}
\noindent (mapping spaces in \w[),]{\C^{\F_{k}}}
with \w{p_{k}} induced by \w[.]{J_{k}^{\ast}} If we assume
$\uY$ is pointed, \wref{eqtower} becomes a tower of pointed simplicial
sets.

Note that \w{J_{k}^{\ast}:\C^{\F_{k+1}}\to\C^{\F_{k}}} has a left
adjoint \w[,]{\xi_{k}:\C^{\F_{k}}\to\C^{\F_{k+1}}} with
\begin{myeq}[\label{eqadjoint}]
(\xi_{k}\uZ)(G/H)~:=~
\raisebox{-0.5ex}{$\substack{{\mbox{{$\colim$}}}\\
{{(\psi\op:G/K\to G/H)}~\text{in}~\EH{k}}}$}~
\uZ(G/K)_{\psi}
\end{myeq}
\noindent at \w{G/H\in\F_{k+1}} for any \w[.]{\uZ:\F_{k}\to\C} In
particular, if \w[,]{G/H\in \F_{k}} the indexing slice category
\w{\EH{k}} has a terminal object \w[,]{\Id:G/H\to G/H} so
\w[.]{(\xi_{k}\uZ)(G/H)=\uZ(G/H)} The counit of the adjunction will be
denoted by \w{\eta_{k}:\xi_{k}J_{k}^{\ast}\uW\to\uW} for
\w[.]{\uW:\F_{k+1}\to\C}
\end{mysubsection}

\begin{prop}\label{pcofib}
When \w{\uX\in\TOG} is a cellular diagram, then for each \w[,]{0<k\leq N}
\w{\tau_{k}\uX} is cofibrant in \w[,]{\TF{k}} and the counit
\w{\eta_{k-1}:\xi_{k-1}\tau_{k-1}\uX\to\tau_{k}\uX} is a
cofibration in \w[.]{\TF{k}}
\end{prop}

\begin{proof}
We can filter $\uX$ by sub-cellular diagrams in \w[:]{\TOG}
$$
\Q{0}\uX~\hra~\Q{1}\uX~\hra~\dotsc~\Q{k}\uX~\hra\dotsc \Q{N}\uX=\uX~,
$$
\noindent where \w{\Q{k}\uX} consists of all cells \w{\uD^{n}_{G/H}}
with \w{G/H\in\F_{k}} (see \wref[).]{eqncell} Then
\w{\tau_{k}\uX=\Q{k}\tau_{k}\uX=\tau_{k}\Q{k}\uX} is a cellular diagram in
\w[,]{\TF{k}} so it is cofibrant, for all \w[.]{k\geq 0}

To prove the last statement, we shall show that for cellular $\uX$,
\w{\xi_{k-1}\tau_{k-1}\uX} is isomorphic to \w[,]{\Q{k-1}\uX} and the
counit \w{\eta_{k-1}} is just the cellular inclusion
\w[.]{\Q{k-1}\uX\hra\Q{k}\uX} For this, by definition of the adjunction
counit it suffices to show that the restriction map
%
%\begin{myeq}[\label{eqrestrict}]
$$
\Hom_{\TF{k}}(\Q{k-1}\uX,\uY)~\xra{\rho}~
\Hom_{\TF{k-1}}(\tau_{k-1}\uX,\,\tau_{k-1}\uY)~.
$$
%\end{myeq}
%
\noindent is a natural isomorphism for any \w[:]{\uY\in\TF{k}} in
other words, that any map of \ww{\F_{k-1}}-diagrams
\w{\uh:\tau_{k-1}\uX\to\tau_{k-1}\uY} extends uniquely to a map of
\ww{\F_{k}}-diagrams \w[.]{\uth{}:\Q{k-1}\uX\to\uY}
We show this by induction on the cellular skeleta
\w[:]{(\uX^{p})_{p=0}^{\infty}}

For \w[,]{p=0} we have
$$
\uX^{0}=\coprod_{\alpha\in I_{0}}\,\uD^{n_{\alpha}}_{G/H_{\alpha}}\hsp\text{and}\hsp
\Q{k-1}\uX^{0}=
\coprod_{\alpha\in I_{0},~G/H_{\alpha}\in\F_{k-1}}~\uD^{n_{\alpha}}_{G/H_{\alpha}}~.
$$
Therefore, for \w[,]{G/H_{\alpha}\in\F_{k-1}}
$$
\uh(G/H^{\alpha}):\tau_{k-1}\uX(G/H^{\alpha})~\lora~\uY(G/H^{\alpha})
$$
\noindent is well-defined, and determines
\w[,]{\uth{0}\rest{\uD^{n_{\alpha}}_{G/H_{\alpha}}}} and thus
\w{\uth{0}:\Q{k-1}\uX^{p}\to\uY} on the coproduct, by Fact \ref{funique}.

For the induction step, assume we have the following
solid diagram, and wish must define the unique \w{\uth{p}:\Q{k-1}\uX^{p}\to\uY}
making the full diagram commute:
\mydiagram[\label{eqinduct}]{
\Q{k-1}\uX^{p-1} \ar@{^{(}->}[rr] \ar[d]^{\rst} \ar@/^{2.5pc}/[rrrr]^{\uth{p-1}}&&
\Q{k-1}\uX^{p} \ar[d]^{\rst} \ar@{.>}[rr]^{\uth{p}}  &&
\uY \ar[d]^{\rst}\\
\tau_{k-1}\uX^{p-1} \ar@{^{(}->}[rr] \ar@/_{2.5pc}/[rrrr]_{\uh^{p-1}}&&
\tau_{k-1}\uX^{p}\ar[rr]_{\uh^{p}} && \tau_{k-1}\uY
}
\noindent where \w[.]{\uh^{i}:=\uh\rest{\tau_{k-1}\uX^{i}}}
Note that the diagram maps here have different indexing categories.

By definition of the cellular skeleta, the commuting square in the
following diagram (in \w[)]{\TF{k}} is cocartesian, so to define the
dotted map \w[,]{\uth{p}:\Q{k-1}\uX^{p}\to\uY} we need only produce a map
\w{g:\coprod_{\alpha\in I_{p},~G/H_{\alpha}\in\F_{k-1}}
\uD^{n_{\alpha}}_{G/H_{\alpha}}\to\uY} making the solid diagram commute:
%
%\mydiagram[\label{eqpushoutcw}]{
$$
\xymatrix{
\ar @{} [drr] |<<<<<<<<<<<<<<<<{\framebox{\scriptsize{PO}}}
\coprod_{\alpha\in I_{p},~G/H_{\alpha}\in\F_{k-1}}~
\uS^{n_{\alpha}}_{G/H_{\alpha}} \ar@{^{(}->}[d]
\ar[rr]^{\coprod_{\alpha} f_{\alpha}} && \Q{k-1}\uX^{p-1}\ar[d]
\ar@/^{2.5pc}/[rrdd]^{\uth{p-1}} \\
\coprod_{\alpha\in I_{p},~G/H_{\alpha}\in\F_{k-1}}~
\uD^{n_{\alpha}}_{G/H_{\alpha}} \ar[rr] \ar@/_{2.5pc}/[rrrrd]_{g}
&& \Q{k-1}\uX^{p} \ar@{.>}[rrd]^{\uth{p}}  && \\
&&&& \uY
}
$$
\noindent Since all the disc-summands in the lower left corner are
indexed by objects \w[,]{G/H_{\alpha}\in\F_{k-1}} as above
\w{\uh(G/H^{\alpha})} determines $g$ by Fact \ref{funique}, which also
ensures that $g$ and \w{\uth{p-1}} agree on the sphere-summands, using
commutativity of \wref[.]{eqinduct}
\end{proof}

\begin{corollary}\label{ccofb}
If $\bX$ is a $G$-CW complex, and \w[,]{\uX=\Phi\bX} each \w{\tau_{k}\uX}
is cofibrant and  \w{\eta_{k-1}:\xi_{k-1}\tau_{k-1}\uX\to\tau_{k}\uX}
is a cofibration.
\end{corollary}

\begin{proof}
Use Assumption \ref{agcw} and Fact \ref{fgcw}.
\end{proof}

\begin{defn}\label{dfiber}
For $\uX$ as above, let \w{\uC_{k}} denote the (homotopy) cofiber of
\w{\eta_{k-1}:\xi_{k-1}\tau_{k-1}\uX\to\tau_{k}\uX} in
\w[,]{\TF{k}} with \w{s_{k}:\tau_{k}\uX\to\uC_{k}} the
structure map in the cofibration sequence. For any fibrant
\w[,]{\uY\in\TOG} define:
%
%\begin{myeq}[\label{eqfiber}]
$$
F_{k}(\uX,\uY)~:=~\map_{\TF{k}}(\uC_{k},\,\tau_{k}\uY)~.
$$
%\end{myeq}
%
\end{defn}

\begin{corollary}\label{ccofib}
For $X$ as above and any fibrant pointed \w[,]{\uY\in\C^{\OG\op}}
we have a fibration sequence of fibrant simplicial sets:
\begin{myeq}[\label{eqfibseq}]
F_{k}(\uX,\uY)~\xra{s^{\ast}_{k}}~\map(\tau_{k}\uX,\,\tau_{k}\uY)~
\xra{p_{k}}~\map(\tau_{k-1}\uX,\,\tau_{k-1}\uY)
\end{myeq}
\end{corollary}

\begin{remark}\label{rtower}
Under the assumptions of the Corollary, \wref{eqtower} is a tower of
fibrations (of Kan complexes), whose (homotopy) limit is the function
complex \w{M=\map_{\TOG}(\uX,\uY)} we are interested in.
Its homotopy spectral sequence thus converges to the homotopy groups
of $M$. In order to make use of it, we need to identify the homotopy
groups of the successive fibers \w[.]{F_{k}}
\end{remark}

\begin{defn}\label{dcofiber}
If $\bX$ is a $G$-CW complex and \w{H\leq G} is any subgroup let
\w{\bX_{H}:=\bigcup_{H<K}\,\bX^{K}} denote the union of the fixed point
sets under all larger subgroups, which is a sub-$\WH$-complex of \w[.]{\bX^{H}}
The quotient \ww{\WH}-space \w{\XHH:=\bX^{H}/\bX_{H}} will be called the
\emph{modified} $H$-fixed point set of $\bX$, with
\w{x^{H}_{0}:=[\bX_{H}]}  as its base point, and \w{s_{H}:\bX^{H}\to\XHH}
the \ww{\WH}-equivariant quotient map.
\end{defn}

\begin{fact}\label{ffree}
The automorphism group \w{\WH} fixes \w[,]{x^{H}_{0}} and acts freely
elsewhere in \w[.]{\XHH}
\end{fact}

\begin{defn}\label{dsubweyl}
If \w[,]{H<K\leq G} let:
$$
\WKH:=(N_{G}K\cap N_{G}H)/H\leq\WH\hsp \text{and}\hsp
\bWHK:=(N_{G}K\cap N_{G}H)/(K\cap N_{G}H)\leq\WK~,
$$
\noindent with a surjective homomorphism \w[.]{\pi:\WKH\to\bWHK}
\end{defn}

\begin{defn}\label{dposet}
Let \w{\Lambda^{H}_{m}} denote the opposite category of the
partially-ordered set of subgroups $K$ of $G$ with
\w[.]{H\leq K\in\hF_{m}} It embeds in \w{\OG\op/(G/H)} by
$$
K~\longmapsto~(i^{\ast}:G/K\to G/H)~.
$$
\end{defn}

Recall that a \emph{skeleton} of a category $\C$ is any full
subcategory \w{\sk\C} whose objects consist of one representative for each
isomorphism type (cf.\ \cite[IV, \S 4]{MacLC}). Any two skeleta of
$\C$ are isomorphic (and equivalent to $\C$).

\begin{lemma}\label{lskeleton}
For every \w[,]{H\leq G} \w{\Lambda^{H}_{m}} is a skeleton of
\w{\EH{m}} (cf.\ \S \ref{sfilter}).
\end{lemma}

\begin{proof}
By Fact \ref{fcommute}, any \w{\psi\op:G/K\to G/H} in
\w{\EH{m}} can be factored uniquely as an isomorphism
\w[,]{\phi:G/K'\to G/K} followed by of \w{i^{\ast}:G/K'\to G/H}
(induced by the inclusion  \w[).]{i:H\hra K'} Thus \w{\psi\op:G/K\to G/H}
is isomorphic in \w{\OG\op/(G/H)} to a unique object
\w{i^{\ast}:G/K'\to G/H} by the unique (vertical) isomorphism:
\mydiagram[\label{eqisomop}]{
G/K' \ar[rr]^{i^{\ast}} && G/H\\
G/K \ar[u]^{\cong}_{\phi\op}  \ar[rru]_{\psi\op=i^{\ast}\circ\phi\op} &&
}
\end{proof}

\begin{mysubsection}{Groups acting on categories}\label{sgac}
Note that the group \w{\WH} acts on the category
\w{\OG\op/(G/H)} and its subcategories \w{\EH{m}} via functors
\w{\Theta_{\bar{a}}:\EH{m}\to\EH{m}} (one for each \w[),]{\bar{a}\in\WH}
where  \w{\Theta_{\bar{a}}} takes \w{\psi\op:G/K\to G/H}
to \w[.]{\co{H}{a}\circ\psi\op}

Therefore, \w{\Theta_{\bar{a}}} induces a \ww{\WH}-action on each
skeleton \w[.]{\Lambda^{H}_{m}} This takes \w{i^{\ast}:G/K\to G/H} to
\w[,]{(i^{a})^{\ast}:G/K^{a}\to G/H} by \wref[,]{eqisomop} using
commutativity of:
%
%\mydiagram[\label{eqwhaction}]{
$$
\xymatrix{
G/K\ar[rr]^{\co{K}{a}}\ar[d]^{i^{\ast}} &&  G/K^{a}\ar[d]^{(i^{a})^{\ast}} \\
G/H\ar[rr]^{\co{H}{a}} && G/H~,
}
$$
\noindent with \w{\Theta_{\bar{a}}} having the
same effect on \w{i^{\ast}:G/K\to G/H} as \w{\Theta_{\bar{b}}}
if and only if \w[.]{ba^{-1}\in N_{G}H\cap N_{G}K}

In particular, if \w[,]{\bar{a}\in\WKH\subseteq\WH} so
\w[,]{(i^{a})^{\ast}=i^{\ast}} then \w{\Theta_{\bar{a}}}
acts on \w{i^{\ast}} via \w[.]{\pi:\WKH\to\bWHK}
\end{mysubsection}

\begin{prop}\label{pcofiber}
If $\bX$ is a $G$-CW complex, $\uY$ is a fibrant and pointed diagram in
\w[,]{\TOG} and \w[,]{\uX=\Phi\bX} then \w{\uC_{k}(G/H)\cong\XHH} (as a
\ww{\WH}-space) for any \w[,]{H\in\hS_{k}} and :
$$
F_{k}(\uX,\uY)~\cong~\prod_{G/H \in\sk\cS_{k}}~\map_{\WH}(\XHH,\,\uY(G/H))~.
$$
\end{prop}

\begin{proof}
By \wref[,]{eqadjoint} \w{\xi_{k}\tau_{k}\uX} agrees with
\w{\tau_{k}\uX} at all objects \w[,]{G/K\in\F_{k-1}} so
\w{\uC_{k}(G/K)=\ast} for such \w[.]{G/K} Thus
$$
F_{k}(\uX,\uY)~\cong~
\map_{\Top^{\cS_{k}}}(\uC_{k}\rest{\Top^{\cS_{k}}},\uY\rest{\Top^{\cS_{k}}})~.
$$
\noindent Since all maps in \w{\cS_{k}} are isomorphisms, by \S \ref{sstrata},
it suffices to consider only its skeleton \w[.]{\sk\cS_{k}}  Therefore:
\begin{myeq}[\label{eqmapspace}]
F_{k}(\uX,\uY)~\cong~\prod_{G/H \in\sk\cS_{k}}~
\map_{\WH}(\uC_{k}(G/H),\,\uY(G/H))~.
\end{myeq}
\noindent The space of \ww{\cS_{k}}-diagram maps is thus equivalent
to the indicated space of \ww{\WH}-equivariant maps, so we have reduced the
study of the fiber in \wref{eqfibseq} once more to the equivariant
category of topological spaces, but for a different (and simpler!)
finite group.

We see from \wref{eqadjoint} that for each \w[,]{G/H\in\cS_{k}} the
cofiber \w{\uC_{k}(G/H)} appearing in \wref{eqmapspace} is the
(homotopy) colimit of the following diagram \w[:]{\uW:J\to\Top}

The forgetful functor \w{U:\EH{k-1}\to\OG\op} sends the object
\w{\phi\op:G/K\to G/H} to \w[,]{G/K} so \w{G/H} is a cocone for the
image subcategory \w[,]{U(\EH{k-1})\subseteq\OG\op} and the indexing
category $J$ is:
\mydiagram[\label{eqi}]{
U(\EH{k-1}) \ar[r]^<<<<{\nu} \ar[d] & G/H &\\
\{\pt\} &&,
}
\noindent where each \w{G/K\in U(\EH{k-1})} is equipped with a map
\w{\nu_{K}:G/K\to G/H} (coming from \w[).]{\EH{k-1}}
The category $J$ has an \ww{\WH}-action by endofunctors as above, and
the functor $\uW$ is induced by the given $\uX$ is \ww{\WH}-equivariant.

For a $G$-CW complex $\bX$, all maps in \w{\uW\rest{U(\LH{k-1})}} are
cofibrations (inclusions \w{\bX^{L}\hra\bX^{K}} for
\w[),]{H\leq K\leq L} and in fact the (homotopy) colimit of the diagram
\w{\uW\rest{U\EH{k-1}}} is just \w[,]{\bX_{H}=\bigcup_{H<K}\,\bX^{K}}
with the \ww{\WH}-action induced by the action on \w[.]{\bX^{H}}

Furthermore, each \w{\uW(\nu_{K}):\bX^{K}\hra\bX^{H}} is a cofibration,
inducing together a cofibration \w[.]{\uW(\nu):\bX_{H}\hra \bX^{H}}
Therefore, the (homotopy) colimit \w{\uC_{k}(G/H)} of $\uW$ over \wref{eqi} is
the (homotopy) cofiber of \w[,]{\uW(\nu)} i.e., \w[,]{\XHH} again with
the obvious \ww{\WH}-action.
\end{proof}

%
%c3   Cohomology with local coefficients
%
\sect{Cohomology with local coefficients}
\label{clocal}
For a group $\Gamma,$ let \w{\EG} denote any contractible space with a free
$\Gamma$-action, and \w{\BG=\EG/\Gamma} the classifying space of
$\Gamma$.

\begin{defn}\label{dborel}
For any $\Gamma$-space $\bX$, the \emph{associated free $\Gamma$-space} is
\w{\EX{\Gamma}{\bX}} with diagonal $\Gamma$-action. The \emph{Borel
construction} on $\bX$ is the (homotopy) quotient
\w{\Br{X}{\Gamma}:=\Bor{\bX}{\Gamma}} \wwh that is, the orbit
space of \w[.]{\EX{\Gamma}{\bX}}

When $\bX$ has a point \w{x_{0}} fixed by $\Gamma$, the
\emph{associated pointed free $\Gamma$-space} is the quotient
$$
\EXp{\Gamma}{\bX}~:=~\EX{\Gamma}{\bX}/\EX{\Gamma}{\{x_{0}\}}~,
$$
\noindent  with $\Gamma$-action induced from diagonal action on
\w[.]{\EX{\Gamma}{\bX}} The \emph{pointed Borel construction} on $\bX$
is
$$
\Brp{\bX}{\Gamma}~:=~\Borp{\bX}{\Gamma}~=~\EXp{\Gamma}{\bX}/\Gamma~.
$$
\end{defn}

Note that \w[,]{\EG\times_{\Gamma}\{\ast\}\simeq\BG} and the constant map
\w{\bX\to\{\ast\}} thus induces a natural map
\w[.]{p:\Bor{\bX}{\Gamma}\to\BG} In particular, for a $\Gamma$-module $M$,
%\begin{myeq}\label{eqgem}
$$
\KGM{n}~:=~\EG\times_{\Gamma}\KM{n}~,
$$
%\end{myeq}
%
\noindent is called the \emph{twisted Eilenberg-Mac~Lane space} (cf.\
\cite{GitC}).  The canonical map \w{p:\KGM{n}\to \BG} has homotopy
fiber \w[.]{\KM{n}}

\begin{defn}\label{dcohloc}
Let $M$ be a $\Gamma$-module. The \emph{$n$-th cohomology group with
local coefficients} of any space \w{\theta:\bX\to \BG} over
\w{\BG} is defined by:
\begin{myeq}\label{eqdloccoh}
H^{n}_{\Gamma}(\bX;M):=\pi_{0}\map_{\BG}(\bX,\,\KGM{n})~.
\end{myeq}
\noindent In particular, any homomorphism \w{\tau:\pi_{1}\bX\to\Gamma}
allows us think of $M$ as a \ww{\pi_{1}\bX}-module; composing
\w{\bB \tau} with the canonical map \w{\bX\to\bB\pi_{1}\bX} shows that
this definition generalizes the original notion of cohomology with a
local coefficient system (see \cite{SteH}).

The mapping space on the right side of \wref{eqdloccoh} is defined by the
pullback diagram:
%
%\mydiagram[\label{eqmapbg}]{
$$
\xymatrix{
\map_{\BG}(\bX,\KGM{n})\ar[r]\ar[d] &
\map(\bX,\KGM{n})\ar[d]^{p_{\ast}} \\
\{\theta\}\ar[r] & \map(\bX,\BG)~.
}
$$
\end{defn}

Note further that by \cite[VI, \S 4]{GJarS} we have a natural isomorphism:
\begin{myeq}\label{eqborel}
\map_{\Gamma}(\EX{\Gamma}{\bX},\KM{n})~\simeq~
\map_{\BG}(\Bor{\bX}{\Gamma},\,\KGM{n})~.
\end{myeq}

\begin{defn}\label{dredloc}
Let $\bX$ be a $\Gamma$-space and $M$ a $\Gamma$-module. The
corresponding $n$-th \emph{reduced} cohomology group with coefficients
in $M$ is defined:
%
%\begin{myeq}\label{eqredloccoh}
$$
\tH{n}{\Gamma}{\Bor{\bX}{\Gamma}}{M}~:=~
\pi_{0}\map_{(\Top/\BG)_{\ast}}(\Bor{\bX}{\Gamma},\,\KGM{n})
$$
%\end{myeq}
%
\noindent (see \cite[Lemma 4.13]{GJarS}). Here \w{(\Top/\BG)_{\ast}}
is the pointed over-category, with objects \w{u:\bY\to \BG}
equipped with a splitting \w[,]{\sigma:\BG\to\bY} for $u$. The mapping
spaces in this category from the object
\w{p :\Bor{\bX}{\Gamma}\to \BG} with splitting \w{s:\BG\to\Bor{\bX}{\Gamma}}
is defined by the pullback diagram:
$$
\xymatrix@R=25pt{\ar @{} [dr] |<<<{\framebox{\scriptsize{PB}}}
\map_{(\Top/\BG)_{\ast}}(\Bor{\bX}{\Gamma},\bY)\ar[r]\ar[d] &
\map_{\BG}(\Bor{\bX}{\Gamma},\bY)\ar[d]^{s^{\ast}} \\
\{\sigma\}\ar[r] & \map_{\BG}(\BG,\bY)~.
}
$$
\noindent
\end{defn}

\begin{prop}\label{ploccoh}
If $\bX$ is a pointed $\Gamma$-CW complex (with base point \w{x_{0}}
fixed by $\Gamma$), and $M$ is a $\Gamma$-module,
then for any \w{0\leq i\leq n} we have:
\begin{myeq}\label{eqmapb}
\pi_{i}\map_{\Gamma}(\EXp{\Gamma}{\bX},\,\KM{n})~\cong~
\tH{n-i}{\Gamma}{\Bor{\bX}{\Gamma}}{M}~.
\end{myeq}
\end{prop}

\begin{proof}
We have a cofibration sequence of $\Gamma$-spaces:
\mydiagram[\label{eqcofibseq}]{
\EX{\Gamma}{\{x_{0}\}}~~\ar@{^{(}->}[rr]^{j} &&
\EX{\Gamma}{\bX} \ar@{->>}[rr]^{s} \ar@/_{1.5pc}/[ll]_{\varphi} &&
\EXp{\Gamma}{\bX}
}
\noindent with \w[.]{\varphi\circ j=\Id}

By \wref{eqborel} and \wref{eqdloccoh} we have
\begin{myeq}\label{eqgeqcoh}
\begin{split}
\pi_{i}\map_{\Gamma}(\EX{\Gamma}{\bX},\KM{n})~\cong&~
H^{n-i}_{\Gamma}(\Bor{\bX}{\Gamma};\,M)\hsp \text{and}\\
\pi_{i}\map_{\Gamma}(\EX{\Gamma}{\{x_{0}\}},\KM{n})~\cong&~
H^{n-i}_{\Gamma}(\BG,M)~,
\end{split}
\end{myeq}
\noindent for all \w[,]{0\leq i\leq n}
since \w[.]{\BG:=\Bor{\{x_{0}\}}{\Gamma}}

Applying \w{\map_{\Gamma}(-,\bK)} for \w{\bK=\KM{n}} to the top
cofibration sequence in \wref{eqcofibseq} yields a split fibration
sequence of simplicial abelian groups:
$$
\xymatrix@R=15pt{
\map_{\Gamma}(\EXp{\Gamma}{\bX},\bK)\ar[r]^{s^{\ast}}&
\map_{\Gamma}(\EX{\Gamma}{\bX},\bK)\ar[r]_<<<<{j^{\ast}}&
\map_{\Gamma}(\EX{\Gamma}{\{x_{0}\}},\bK)\ar@/_{2.0pc}/[l]_{\varphi^{\ast}}
}
$$
\noindent and thus (using \wref[),]{eqgeqcoh}
a split short exact sequence of homotopy groups:
$$
\xymatrix{
0\to\pi_{i}\map_{\Gamma}(\EXp{\Gamma}{\bX},\,\KM{n})\ar[r]^<<<<{s^{\ast}}
&
\HL{n-i}{\Gamma}{\Bor{\bX}{\Gamma}}{M}\ar[r]_<<<<<{j^{\ast}} &
\HL{n-i}{\Gamma}{\BG}{M}\to 0
\ar@/_{1.5pc}/[l]_{\varphi^{\ast}}~.
}
$$
\noindent On the other hand, we know that the last two terms fit into
a split short exact sequence:
\begin{myeq}\label{eqsseseq}
0\to \tH{n-i}{\Gamma}{\Bor{\bX}{\Gamma}}{M}~\to~
\HL{n-i}{\Gamma}{\Bor{\bX}{\Gamma}}{M}~\xra{j^{\ast}}~
\HL{n-i}{\Gamma}{\BG}{M}~\to 0
\end{myeq}
\noindent (see \cite[VI, \S 4]{GJarS}), so there is a natural
isomorphism \wref{eqmapb} for each \w[.]{0\leq i\leq n}
\end{proof}

\begin{corollary}\label{cptcofseq}
If \w{\bX\xra{f}\bY\xra{g}\bZ} is a cofibration sequence of pointed
$\Gamma$-CW complexes and $M$ is a $\Gamma$-module, we have a long
exact sequence in reduced cohomology with local coefficients:
\begin{equation*}
\begin{split}
\dotsc~\tH{n}{\Gamma}{\Bor{\bZ}{\Gamma}}{M}~&\xra{g^{\ast}}~
\tH{n}{\Gamma}{\Bor{\bY}{\Gamma}}{M}~\xra{f^{\ast}}~
\tH{n}{\Gamma}{\Bor{\bX}{\Gamma}}{M}\\
~&\xra{\delta}~\tH{n+1}{\Gamma}{\Bor{\bZ}{\Gamma}}{M}~\dotsc .
\end{split}
\end{equation*}
\end{corollary}

\begin{proof}
The functorial colimit \w{\EXp{\Gamma}{-}} preserves cofibration sequences, so
\begin{myeq}\label{eqnewptcofseq}
\EXp{\Gamma}{\bX}~\xra{f_{\ast}}~\EXp{\Gamma}{\bY}~\xra{g_{\ast}}~\EXp{\Gamma}{\bZ}
\end{myeq}
\noindent is a homotopy cofibration sequence of pointed $\Gamma$ spaces.
Applying \w{\map_{\Gamma}(-,\KM{n})} to \wref{eqnewptcofseq} yields a
fibration sequence in \w[,]{\Sa} whose long exact sequence in homotopy
groups is the required one, by Proposition \ref{ploccoh}.
\end{proof}

\begin{lemma}\label{lfreeptact}
If $\bX$ is a $\Gamma$-CW complex with \w{x_{0}} fixed by $\Gamma$, and free
action on \w[,]{\bX\setminus\{x_{0}\}} and $M$ is a $\Gamma$-module,
then for any \w[:]{0\leq i\leq n}
$$
\pi_{i}\map_{\Gamma}(\bX,\KM{n})~\cong~
\tH{n-i}{\Gamma}{\Bor{\bX}{\Gamma}}{M}~.
$$
\end{lemma}

\begin{proof}
In our case the diagram \wref{eqcofibseq} fits into a diagram of $\Gamma$-spaces:
\mydiagram[\label{eqmapcofibseq}]{
\EX{\Gamma}{\{x_{0}\}}~~\ar@{^{(}->}[rr]^{j} \ar[d]^{p}_{\he} &&
\EX{\Gamma}{\bX} \ar@{->>}[rr]^{s} \ar[d]^{q}_{\he} \ar@/_{1.5pc}/[ll]_{\varphi} &&
\EXp{\Gamma}{\bX} \ar[d]^{r} \\
\{x_{0}\}~~\ar@{^{(}->}[rr] && \bX \ar@{->}[rr]^{\Id} && \bX
}
\noindent where the vertical maps are projections onto the second
factor; since $p$ and $q$ are (non-equivariant) homotopy equivalences, so is $r$.

Note that \w{r:\EXp{\Gamma}{\bX}\to\bX} induces homotopy equivalences on
all fixed point sets (which consist only of the basepoint for all
\w[),]{\{e\}\neq H\leq G} so by \cite[Theorem (1.1)]{JSegaE} $r$ is in fact a
$\Gamma$-homotopy equivalence. Therefore, it induces a weak
equivalence:
\begin{myeq}\label{eqgammaeq}
\map_{\Gamma}(\bX,\KM{n})~\xra{r^{\ast}}~
\map_{\Gamma}(\EXp{\Gamma}{\bX},\KM{n})~,
\end{myeq}
\noindent so the claim follows from Proposition \ref{ploccoh}.
\end{proof}

From Fact \ref{ffree} and Lemma \ref{lfreeptact} we deduce:

\begin{lemma}\label{lloccoh}
For any finite group $G$, $G$-CW complex $\bX$, coefficient system
\w[,]{\uM:\OG\op\to\Abgp} and \w[,]{n>0} let $\uX$ be the fixed point
set diagram for $\bX$, and \w{\uY:=\uK(\uM,n)}
the (fibrant) diagram of Eilenberg-Mac~Lane spaces in
\w{\TOG} corresponding to $\uM$. Then for each
\w{k\leq\lenG\{e\}} we have
$$
\pi_{i}F_{k}(\uX,\uY)~\cong~\bigoplus_{G/H\in\sk\cS_{k}}\
\tH{n-i}{\WH}{\XWH}{\MH}
$$
\noindent for \w[,]{0\leq i<n} where \w{F_{k}(\uX,\uY)} is as in
Definition \ref{dfiber} and\w[.]{\MH:=\uM(G/H)}
\end{lemma}

\begin{defn}\label{dcofibseq}
Given a $G$-CW complex $\bX$ and fixed \w[,]{G/H\in\cS_{k+1}} we have a
diagram of \ww{\WH}-spaces:
\mydiagram[\label{eqthreebythree}]{
\bigcup_{H<L\in\hF_{k}}~\bX_{L} \ar@{^{(}->}[rr]^{=} \ar@{^{(}->}[d] &&
\bigcup_{H<L\in\hF_{k}}~\bX_{L} \ar[rr] \ar@{^{(}->}[d] &&
\{\ast\} \ar@{^{(}->}[d] \\
\bX_{H}=\bigcup_{H<L\in\hF_{k}}~X^{L} \ar@{^{(}->}[rr] \ar@{->>}[d]_{q} &&
\bX^{H} \ar@{->>}[rr] \ar@{->>}[d] && \XHH \ar[d]^{=} \\
\bigvee_{H<L\in\hS_{k}}~\XLL \ar@{^{(}->}[rr] &&
\bX^{H}/\left(\bigcup_{H<L\in\hF_{k}}~\bX_{L}\right) \ar@{->>}[rr] && \XHH
}
\noindent in which all rows and columns are cofibration sequences.

Denote the connecting map for the bottom row in \wref{eqthreebythree}
(a cofibration sequence of \emph{pointed} \ww{\WH}-spaces) by
\begin{myeq}[\label{eqconnmap}]
\XHH~\xra{\bar{\delta}_{H}}~\Sigma \bigvee_{H<L\in\hS_{k}}~\XLL~.
\end{myeq}
\noindent By Corollary \ref{cptcofseq} we have a long exact sequence in reduced
cohomology with coefficients in a \ww{\WH}-module $M$, with \emph{connecting homomorphism}:
\begin{myeq}[\label{eqconnhom}]
\tH{i}{\WH}{\Borp{(\bigvee_{H<L\in\hS_{k}}\XLL)}{\WH}}{M}~
\xra{\bar{\delta}^{\ast}_{H}}~\tH{i+1}{\WH}{\Borp{\XHH}{\WH}}{M}~.
\end{myeq}
\end{defn}

\begin{remark}\label{rchangegroup}
If $\Gamma$ acts on a (pointed) space $\bY$ as a
  subgroup of \w[,]{\Gamma'} we can replace \w{\bE\Gamma} by
  \w{\bE\Gamma'} (which is still a contractible space with a free
  $\Gamma$ action) in
  constructing \w[,]{\EX{\Gamma}{\bY}} \w[,]{\Bor{\bY}{\Gamma}}
   \w[,]{\EXp{\Gamma}{\bY}} and \w[.]{\Borp{\bY}{\Gamma}}
\end{remark}

\begin{mysubsection}{The change of groups map}\label{schangegp}
Given a $G$-CW complex $\bX$ and \w[,]{H\in\hF_{k+1}} we have a homeomorphism
of pointed spaces:
\begin{myeq}[\label{eqequivalence}]
\EXp{\WH}{(\bigvee_{H<L\in\hS_{k}}\XLL)}~=~
\bigvee_{H<L\in\hS_{k}}~\EXp{\WH}{\XLL}~.
\end{myeq}
\noindent The \ww{\WH}-action on the left-hand side transfers to the
right-hand side as follows: for each \w{L>H} in \w[,]{\hS_{k}} any
element \w{\bar{a}\in\WH} takes \w{\XLL} to \w{\bX^{L^{a}}_{L^{a}}} under left
multiplication by $a$. In particular, \w{\bar{a}\in\WLH\leq\WH} acts on
\w{\XLL=\bX^{L^{a}}_{L^{a}}} by automorphisms (via \w{\bWHL} \wwh
cf.\ \S \ref{dsubweyl}).

For any $G$-module $\uM$, we define the \emph{change of groups map}:
%
%\begin{myeq}[\label{eqchange}]
$$
\Phi^{L}_{H}:\map_{\WL}(\EXp{\WL}{\XLL},\,\bK(M_{L},n))\to
\map_{\WH}(\EXp{\WH}{(\bigvee_{H<L\in\hS_{k}}\XLL)},\,\bK(M_{H},n))
$$
%\end{myeq}
%
\noindent to be the following composite:
\begin{myeq}\label{eqcomposite}
\begin{split}
&\map_{\WL}(\EXp{\WL}{\XLL},\bK(M_{L},n))~\xra{i}~
\map_{\bWHL}(\EXp{\WL}{\XLL},\bK(M_{L},n))~\xra{\simeq}\\
&~\map_{\bWHL}(\EXp{\bWHL}{\XLL},\bK(M_{L},n))~\xra{\pi^{\ast}}~
\map_{\WLH}(\EXp{\WLH}{\XLL},\bK(M_{L},n))~\xra{\mu_{\ast}}\\
&~\map_{\WLH}(\EXp{\WLH}{\XLL},\bK(M_{H},n))~\xra{\simeq}~
\map_{\WLH}(\EXp{\WH}{\XLL},\bK(M_{H},n))~\xra{\eta}\\
&~\map_{\WH}(\EXp{\WH}{(\bigvee_{H<L'\in\hS_{k}}\XLLp)},\,\bK(M_{H},n))~.
\end{split}
\end{myeq}
\noindent Here $i$ is an inclusion (since \w[),]{\bWHL\leq\WL} the
homotopy equivalences exist by Remark \ref{rchangegroup},
\w{\pi_{\ast}} is induced from \w{\pi:\WLH\to\bWHL} by functoriality of
\w[,]{\Gamma\mapsto\bE\Gamma} and \w{\mu:M_{L}\to M_{H}} is the
$\uM$-structure map.

The map $\eta$ is constructed by an ``averaging'' process, as follows:

First, note that by \wref[,]{eqequivalence} we have:
\begin{myeq}[\label{eqproduct}]
\map_{\WL}(\EXp{\WL}{(\bigvee_{L<H'\in\hS_{k}}\XHHp)},\bK(M_{L},n))
\subseteq\prod_{L<H'\in\hS_{k}}\map_{\WHLp}(\EXp{\WL}{\XHHp},\,\bK(M_{L},n))
\end{myeq}
\noindent so to define $\eta$ we must specify a
\ww{\WHLp}-equivariant map \w{f':\EXp{\WL}{\XHHp}\to \bK(M_{L},n)}
in each mapping space in the right hand side of \wref{eqproduct}
(and verify that the result is indeed \ww{\WL}-equivariant).

For any \w{a\in N_{G}H} and \ww{\WLH}-equivariant map
\w[,]{f\in\map_{\WLH}(\EXp{\WH}{\XLL},\bK(M_{H},n))} define:
$$
f':=f_{a}~\in\map_{\WH}(\EXp{\WH}{X^{L^{a}}_{L^{a}}},\,\bK(M_{H},n))
$$
\noindent (for \w{L':=L^{a}} to be \w[.]{f_{a}:=\mu_{a}\circ f\circ\co{L}{a}}
As in \S \ref{socat},
$$
(\cj{L}{a})^{\ast}:\EXp{\WH}{X^{L^{a}}_{L^{a}}}\to\EXp{\WH}{\XLL}
$$
\noindent is induced by \w{\cj{L}{a}:G/L\to G/L^{a}} and \w{\mu_{a}:M_{H}\to M_{H}}
is induced by the automorphism \w[.]{\cj{H}{a}:G/H\to G/H} It is
readily verified that \w{f_{a}} is \ww{W^{L^{a}}_{H}}-equivariant \wh
that is, that the following diagram commutes:
$$
\xymatrix{
\EXp{\WH}{X^{L^{a}}_{L^{a}}}\ar[r]^{(\cj{L}{a})^{\ast}}
\ar[d]_{\overline{b^{a}}_{\ast}=(\cj{L^{a}}{b^{a}})^{\ast}} &
\EXp{\WH}{\XLL}\ar[r]^{f}\ar[d]_{\overline{b}_{\ast}=(\cj{L}{b})^{\ast}}&
\bK(M_{H},n) \ar[r]^{\mu_{a}} \ar[d]^{\mu_{a}} &
\bK(M_{H},n) \ar[d]^{\mu_{b^{a}}} \\
\EXp{\WH}{X^{L^{a}}_{L^{a}}}\ar[r]_{(\cj{L}{a})^{\ast}} &
\EXp{\WH}{\XLL}\ar[r]_{f}&  \bK(M_{H},n) \ar[r]_{\mu_{a}} & \bK(M_{H},n)
}
$$
\noindent for each \w{b\in N_{G}H\cap N_{G}{L}} (so $b$
represents an automorphism \w{\overline{b}\in\WLH} and \w{b^{a}}
represents \w[).]{\overline{b^{a}}\in W^{L^{a}}_{H}}
Moreover, we can see that \w{f_{a}} coincides with $f$ when \w[.]{a\in N_{G}L}

If \w{L'\in\hS_{k}} which is \emph{not} conjugate to the give $L$ by
an element \w[,]{a\in N_{G}H} we set
\w{f'\in\map_{\WLHp}(\EXp{\WH}{\XLLp},\,\bK(M_{H},n))} equal to zero. The
resulting sequence \w{(f')_{H<L'\in\hS_{k}}} is in fact in the left
hand side of \wref{eqproduct} \wwh that is, it defines a
\ww{\WH}-equivariant map on
\w[,]{\EXp{\WH}{(\bigvee_{H<L'\in\hS_{k}}\XLLp)}} as required.
\end{mysubsection}

From Proposition \ref{ploccoh} we deduce:

\begin{corollary}\label{cchangegp}
Given a $G$-CW complex $X$, a $G$-module $\uM$, and \w[,]{H\in\hF_{k+1}}
we have a \emph{change of groups homomorphism}
$$
(\Phi^{L}_{H})_{\ast}:\tH{\ast}{\WL}{\Bor{\XLL}{\WL}}{M_{L}}\to
\tH{\ast}{\WH}{\Bor{(\bigvee_{H<L'\in\hS_{k}}\XLLp)}{\WH}}{M_{H}}~.
$$
\end{corollary}

%
%c4   The spectral sequences
%
\sect{The spectral sequences}
\label{csss}

Our goal is to understand how the Bredon cohomology of a $G$-CW complex $\bX$
can be computed in terms of \emph{local} information \wh that is, the
fixed point sets \w{X^{H}} for various subgroups \w[.]{H\leq G}

Conceptually, this local data should depend only on the Bredon
cohomology of \w{\bX^{H}} with respect to the action of the automorphism
group \w[.]{\WH} In practice, the local data is more delicate, since
it is expressed in terms of the reduced cohomology groups of
\w{\bX^{H}} with local coefficients. In fact, we have two different spectral sequences, which allow us to compute both Bredon cohomology groups (of $\bX$ itself, and of \w[)]{\bX^{H}} in terms of reduced cohomology with local coefficients.

First, we state our main result:

\begin{thm}\label{tss}
For any finite group $G$, $G$-CW complex $\bX$, and coefficient system
\w[,]{\uM:\OG\op\to\Abgp} there is a first quadrant spectral sequence
with:
$$
E_{1}^{k,i}~=~\bigoplus_{G/L\in\sk\cS_{k+i}}
\tH{i}{\WL}{\XWL}{\ML}~~\Longrightarrow~~\HB{i}{G}{\bX}{\uM}~.
$$
\noindent The differential \w{d_{1}:E_{1}^{k,i}\to E_{1}^{k,i+1}=
\bigoplus_{G/H\in\sk\cS_{k+i+1}}\tH{i+1}{\WH}{\XWH}{\MH}}
on each summand
is non-zero only if \w[,]{H<L} in which case its component
$$
\tH{i}{\WL}{\XWL}{\ML}~\to~\tH{i+1}{\WH}{\XWH}{\MH}
$$
\noindent is the composite of the connecting homomorphism
\w{\delta^{i}} of \wref{eqconnhom} with the change of groups
homomorphism of Corollary \ref{cchangegp}.
\end{thm}

\begin{proof}
Fix \w{n>0} and let \w{\uY:=\uK(\uM,n)} be the (fibrant) diagram of
Eilenberg-Mac~Lane spaces in \w{\TOG} corresponding to $\uM$,
so that \w{\pi_{t}\map_{\TOG}(\uX,\uY)\cong\HB{n-t}{G}{\bX}{\uM}}
for any \w[.]{0\leq i\leq n} If \w[,]{N=\lenG\{e\}} then the tower
\wref{eqtower} is a tower of fibrations (by Corollary \ref{ccofib}),
with \w{\map(\uX,\uY)=\map(\tau_{N}\uX,\tau_{N}\uY)}
as its (homotopy) limit.

The usual homotopy spectral sequence for this tower of fibrations
(cf.\ \cite[IX, \S 4.2]{BKaH}) has:
$$
E_{1}^{k,t}~:=~\pi_{t}F_{k-t}(\uX,\uY)~\Rw~
\pi_{t}\map(\uX,\uY)~\cong~\HB{n-t}{G}{\bX}{\uM}~.
$$
\noindent By Lemma \ref{lloccoh} we have:
$$
\pi_{t}F_{k-t}~\cong~
\bigoplus_{G/L \in \sk\cS_{k-t}}\tH{n-t}{\WL}{\XWL}{\ML}~.
$$
\noindent Finally, set \w[,]{t:=n-i} and note that if
we replace $n$ by \w[,]{n-1} we simply apply the loop functor
\w{\Omega} to the diagram $\uY$, and thus to the tower of fibrations
\wref[.]{eqtower} Therefore, the spectral sequences for different $n$
all agree\vsm .

The differential \w{d_{1}^{k-i,i}:E_{1}^{k-i,i}\to E_{1}^{k-i+1,i+1}}
is induced from the long exact sequences in \w{\pi_{\ast}} of the
fibration sequences \wref{eqfibseq} for $k$ and \w{k+1} by the
composite:
\begin{myeq}\label{eqdifferential}
\begin{split}
\pi_{t}\map_{\TF{k}}(\uC_{k},\tau_{k}\uY)~&\xra{s_{k}^{\ast}}~
\pi_{t}\map_{\TF{k}}(\tau_{k}\uX,\tau_{k}\uY))~\cong~
\pi_{t}\map_{\Top^{\F_{k+1}}}(\xi_{k}\tau_{k}\uX,\tau_{k+1}\uY))\\
&~\xra{\partial_{t}}~
\pi_{t-1}\map_{\Top^{\F_{k+1}}}(\uC_{k+1},\tau_{k+1}\uY))~,
\end{split}
\end{myeq}
\noindent where the first group is:
\begin{myeq}\label{eqsource}
\bigoplus_{G/L\in\sk\cS_{k}}~\tH{i}{\WL}{\XWL}{\ML}
\end{myeq}
\noindent and the last group is:
\begin{myeq}\label{eqtarget}
\bigoplus_{G/H\in\sk\cS_{k+1}}~\tH{i+1}{\WH}{\XWH}{\MH}
\end{myeq}
\noindent both of which are \emph{finite} direct sums of abelian
groups.

To define a homomorphism into the direct sum (product), it suffices to
describe its component in each summand of \wref{eqtarget}
indexed by \w[:]{G/H\in\sk\cS_{k+1}}

First, note that the connecting homomorphism \w{\partial_{t}} in the
fibration sequence \wref{eqfibseq} for \w{k+1} is induced by the
connecting map \w{\delta_{k+1}} in the cofibration sequence:
%
%\begin{myeq}\label{eqcofibration}
$$
\xi_{k}\tau_{k}\uX~\xra{\eta_{k}}~\tau_{k+1}\uX~\xra{s_{k+1}}~\uC_{k+1}~
\xra{\delta_{k+1}}~\Sigma\xi_{k}\tau_{k}\uX
$$
%\end{myeq}
%
\noindent in \w[.]{\Top^{\F_{k+1}}} Thus given a map of diagrams
\w[,]{\uf:\Sigma\xi_{k}\tau_{k}\uX\to\tau_{k+1}\uY} we need to
calculate the composite
\w[.]{\uf\circ\delta_{k+1}:\uC_{k+1}\to\tau_{k+1}\uY}
However, the diagram \w{\uC_{k+1}} is trivial except on
\w[,]{\cS_{k+1}} so this map is completely determined by
\w{(\uf\circ\delta_{k+1})(G/H)=\uf(G/H)\circ\delta_{k+1}(G/H)} for
\w[,]{G/H\in\cS_{k+1}} where \w{\delta_{k+1}(G/H)} is the connecting
map \w{{\delta}_{H}} in the middle horizontal cofibration sequence in
\wref[.]{eqthreebythree}

By \wref{eqadjoint} we have:
$$
\Sigma\xi_{k}\tau_{k}\uX(G/H)~=~\Sigma\colim\uX(G/L)~=~\colim\Sigma\uX(G/L)~,
$$
\noindent where the colimit, taken over all \w{G/L\to G/H} in \w{\OG\op} for
\w[,]{G/L\in\F_{k}} is \w[.]{\bigcup_{H<L\in\hF_{k}}~\Sigma X^{L}}
Therefore:
\begin{myeq}\label{equnionmap}
\uf(G/H)~=~\bigcup_{H<L\in\hF_{k}}~\uY(i_{\ast}\op)\circ\uf(G/L)
\end{myeq}
\noindent for \w{\uf(G/L):\Sigma\bX^{L}\to\bK(M_{L},n)} and
\w{\uY(\psi)} induced by the map \w{i_{\ast}\op:G/L\to G/H} in \w{\OG\op}
corresponding to the inclusion \w{i:H\hra L} (inducing the map denoted
by \w{\mu_{\ast}} in \wref[).]{eqcomposite}

The decomposition \wref{equnionmap} is clearly unaffected by the
isomorphism in \wref[,]{eqdifferential} and the map \w{s_{k}^{\ast}}
is induced by the quotient map \w{s_{k}} in the cofibration sequence:
%
%\begin{myeq}\label{eqcofibrationtwo}
$$
\xi_{k-1}\tau_{k-1}\uX~\xra{\eta_{k-1}}~\tau_{k}\uX~\xra{\us_{k}}~\uC_{k}
$$
%\end{myeq}
%
\noindent in \w[.]{\TF{k}}

It might seem that we must evaluate \w{s_{k}} at all \w{G/L\in\F_{k}}
with \w[.]{H<L} However, in order to define the homomorphism out of
the direct sum (coproduct) \wref[,]{eqsource} it suffices to describe
its component on each summand, and since we are mapping into the
summand for \w[,]{G/H\in\sk\cS_{k+1}} we need only consider
subgroups \w{L>H} with \w[.]{G/L\in\sk\cS_{k}} and for these
we have: \w{\us_{k}(G/L)=q} (the lower left vertical quotient map in
\wref[).]{eqthreebythree}

We thus see that the composite \wref{eqdifferential} is induced by the
composite:
$$
\XHH~\xra{\delta_{H}}~\Sigma\bX_{H}=\bigcup_{H<L\in\hF_{k}}~\Sigma\bX^{L}~
~\xra{\Sigma q}~\bigvee_{H<L\in\hS_{k}}~\Sigma\XLLp
$$
which is just \w{\bar{\delta}_{H}} of \wref{eqconnmap} by naturality of the
connecting maps in the diagram of cofibration sequences
\wref[.]{eqthreebythree}

At this point we have shown that, given a \ww{\WH}-map
$$
g:\bigvee_{H<L\in\hS_{k}}\Sigma \XLL~\lora~\bK(M_{H},n)~,
$$
\noindent or its adjoint
$$
\tilde{g}:\bigvee_{H<L\in\hS_{k}}\XLL~\lora~\Omega
\bK(M_{H},n)~=~\bK(M_{H},n-1)~,
$$
\noindent we obtain a \ww{\WH}-map
\w[.]{g\circ(\Sigma q\circ\delta_{H}):\XHH\to \bK(M_{H},n)}
By Lemma \ref{lloccoh} this yields a homomorphism:
%
%\begin{myeq}\label{eqhomolocoh}
$$
\tH{i}{\WH}{\Bor{\left(\vee_{H<L\in\hS_{k}}\Sigma \XLL\right)}{\WH}}{M_{H}}~
\xra{(\Sigma q\circ\delta_{L})^{\ast}}~
\tH{i+1}{\WH}{\Bor{\XHH}{\WH}}{M_{H}}~.
$$
%\end{myeq}

Now fix \w{L_{0}>H} in \w{\hS_{k}} \wwh that is, our representative in the
skeleton of \w{\EH{k}} (see Lemma \ref{lskeleton}) \wh and consider
the inclusion of the summand
\begin{equation*}
\begin{split}
\tH{i}{W_{L_{0}}}{\XWLz}{M_{L_{0}}}~&\xra{\iota_{L_{0}}}~
\bigoplus_{G/L\in\sk\cS_{k}}~\tH{i}{\WL}{\XWL}{\ML}\\
\cong&~\pi_{t}\map_{\TF{k}}(\uC_{k},\tau_{k}\uY)
\end{split}
\end{equation*}
\noindent in the direct sum \wref[.]{eqsource}

The composite \w{s_{k}^{\ast}\circ\iota_{L_{0}}} into
\w{\pi_{t}\map_{\TF{k}}(\tau_{k}\uX,\tau_{k}\uY)} is
induced by the equivalence of categories
\w[,]{\tau:\F_{k}/(G/H)\to\sk\F_{k}/(G/H)} which extends
a \ww{W_{L_{0}}}-map \w{g:\XLLz\to\uY(G/L_{0})=\bK(M_{L_{0}},i)}
(representing \w[)]{[g]\in\tH{i}{W_{L_{0}}}{\XWLz}{M_{L_{0}}}} to
all of \w{\TF{k}} precisely by the averaging map $\eta$
in \wref[.]{eqcomposite}
\end{proof}

\begin{mysubsection}{The categories \w{\EH{m}} with group action}\label{sehga}
As noted in \S \ref{sgac}, for a fixed \w[,]{G/H\in\cS_{k}} the group
\w{\WH=\Aut(G/H)} acts on the categories \w[;]{\EH{m}} we can
incorporate this action into each \w{\EH{m}} to obtain a new category
\w{\tEH{m}} defined as follows:

First, we define a free category $\C$ with objects
\w[:]{\Obj(\EH{m})} that is, \w{\psi\op:G/K\to G/H} in \w{\OG\op}
with \w[.]{G/K\in\F_{m}} The maps of $\C$ will be generated by those
of \w[,]{\EH{m}} together with a new map \w{\cp{K}{a}:i^{\ast}\to(i^{a})^{\ast}}
for any \w{a\in N_{G}H} and \w{i^{\ast}:G/K\to G/H} in
\w{\LH{m}\subseteq\Obj(\EH{m})} (cf.\ \S \ref{sgac}).

The forgetful functor \w{U_{m}:\EH{m}\to\OG\op} (sending
\w{\psi\op:G/K\to G/H} to \w[)]{G/K} extends to \w{\tU{m}:\C\to\OG\op}
in the obvious way, with
$$
\tU{m}(\cp{K}{a})~=~\co{K}{a}:G/K\to G/K^{a}~.
$$
\noindent We define \w{\tEH{m}} to be the quotient category of $\C$, with
the same object set, in which two morphisms $f$ and $g$ are
identified if and only if \w[.]{\tU{m}f=\tU{m}g} Thus \w{\EH{m}}
embeds faithfully (but not fully) in \w[.]{\tEH{m}}

This implies that if \w[,]{a\in H\leq K} then \w{\co{K}{a}} is the
identity, so \w{\tEH{m}} encodes the action of \w{\WH =N_{G}H/H} on
\w{\EH{m}} (see \S \ref{sgac}). In particular,
\w{\cp{K}{a}:i^{\ast}\to(i^{a})^{\ast}} will be identified
with \w{\cp{K}{b}:i^{\ast}\to(i^{b})^{\ast}} if
$$
\co{K}{a}=\co{K}{b}:G/K\to G/K^{a}=G/K^{b}~,
$$
\noindent that is, if \w[.]{b^{-1}a\in K} In particular,
\w{N_{G}H\cap N_{G}K} acts by
automorphisms on \w[,]{i^{\ast}:G/K\to G/H} and
$$
\Aut_{\tEH{m}}(i^{\ast})~\cong~(N_{G}H\cap N_{G}K)/(N_{G}H\cap K)~=~
\bWHK~.
$$

The filtration \wref{eqfilter} induces a filtration:
\begin{myeq}[\label{eqfilterd}]
\tEH{0}~\subset~\tEH{1}~\subset~\dotsc~\tEH{m}\subset~\dotsc~\subset~\tEH{k}
\end{myeq}
\noindent by full subcategories. We let \w{\tSH{m}} be the full subcategory
of \w{\tEH{m}} with object set \w[.]{\tEH{m}\setminus\tEH{m-1}}
\end{mysubsection}

\begin{defn}\label{dfiltdiag}
Given a $G$-space $\bX$, \w[,]{\uM:\OG\op\to\Abgp} \w[,]{n\geq 0} and a
fixed subgroup \w[,]{H\leq G} let \w{\bY:=\bK(\uM,n)^{H}} (an
\ww{\WH}-Eilenberg-Mac~Lane space with \w[).]{\pi_{n}\bY=M_{H}:=\uM(G/H)}

We define two diagrams \w{\hX,\hY:\tEH{k}\to\Top} by:
\begin{enumerate}
\renewcommand{\labelenumi}{(\alph{enumi})\ }
\item \w[,]{\hX:=\uX\circ\tU{k}} so \w{\hX(G/K):=\bX^{K}} with
\w{\hX(j^{\ast})} the inclusion \w{\bX^{L}\hra\bX^{K}}
for \w{j:L\hra K} and \w{\hX(\co{K}{a}):\bX^{K}\to\bX^{K^{a}}} the action of
\w[.]{N_{G}H}
\item \w[,]{\hY(G/K):=\bY^{N_{G}H\cap K}} with
\w{\hY(j^{\ast})} the inclusion, and
$$
\hY(\co{H}{a}):\bY^{N_{G}H\cap K}~\lora~\bY^{N_{G}H\cap K^{a}}
$$
\noindent  again the given \ww{N_{G}H}-action.
\end{enumerate}
\end{defn}

\begin{lemma}\label{lmapspace}
If $\bX$ and $\bY$ as above, there is a natural isomorphism
$$
\map_{\TE{k}}(\hX,\hY)~\cong~\map_{\WH}(\bX^{H},\bY)~.
$$
\end{lemma}

\begin{proof}
Note that for any diagram \w[,]{\hZ:\tEH{k}\to\Top} \w{\hZ(G/H)} has a
\ww{\WH}-action, and the restriction \w{\uZ:\EH{k}\to\Top} of $\hZ$ is
\ww{\WH}-equivariant (with respect to the \ww{\WH}-action of \S \ref{sgac}).

Since \w{\hX(G/H)=\bX^{H}} and \w[,]{\hY(G/H)=\bY} we have a projection
$$
\map_{\TE{k}}(\hX,\hY)~\to~\map(\bX^{H},\bY)
$$
\noindent which lands in \w{\map_{\WH}(\bX^{H},\bY)} (because of the
equivariance).

On the other hand, given a \ww{\WH}-map \w{f:\bX^{H}\to\bY} and
\w[,]{i:H\hra K} because \w[,]{N_{G}H\cap K\leq K} we have an
inclusion \w[.]{j:\bX^{K}\hra\bX^{N_{G}H\cap K}} Moreover,
\w{f\rest{\bX^{N_{G}H\cap K}}} lands in \w{\bY^{N_{G}H\cap K}=\hY(G/K)}
because any \ww{\WH}-map is an \ww{N_{G}H}-map. Therefore, we can
define a map of \ww{\tEH{m}}-diagrams \w{\hf:\hX\to\hY} by setting
$$
\hf(G/K):\hX(G/K)\to\hY(G/K)
$$
\noindent to be
$$
f\rest{\bX^{K}}~=~(f\rest{\bX^{N_{G}H\cap K}})\circ j:\bX^{K}~\to~
\bY^{N_{G}H\cap K}~,
$$
\noindent and noting that all the groups acting in \w{\tEH{m}} do so via the
\ww{N_{G}H}-action.
\end{proof}

\begin{defn}\label{dskelehm}
Given an inclusion \w[,]{\ell:H\hra L} let \w{\lra{\ell^{\ast}}} denote the
collection of all distinct conjugates
\w{(\ell^{a})^{\ast}:G/L^{a}\to G/H^{a}=G/H} of \w{\ell^{\ast}:G/L\to G/H} by
elements \w[.]{a\in N_{G}H} This class contains
\w{[N_{G}H:N_{G}H\cap N_{G}L]} distinct elements of \w[.]{\OG\op/(G/H)}
When \w[,]{L\in\hF_{m}} \w{\lra{\ell^{\ast}}} is an isomorphism class of
elements in \w[,]{\tEH{m}} so by choosing one representative
\w{\ell^{\ast}:G/L\to G/H} in each such class \w[,]{\lra{\ell^{\ast}}} we
obtain a skeleton \w{\sk\tEH{m}} for \w[.]{\tEH{m}}

Note that any object \w{\ell^{\ast}:G/L\to G/H} in the skeleton
(with \w{\ell:H\hra L} an inclusion) is determined by
the object \w{G/L} in \w[,]{\OG\op} so by abuse of notation we
denote it simply by \w[.]{G/L}
\end{defn}

\begin{prop}\label{pcellular}
If \w{\uX:\OG\op\to\Top} is cellular, then for each \w{H\in\hS_{k}}
and \w{m\leq k} the restriction
\w{\uX\circ\tU{m}:\tEH{m}\to\Top} is cellular, too.
\end{prop}

\begin{proof}
For any \w{L\in\hF_{m}} and inclusion \w[,]{i:H\hra K} let
$$
\cV^{L}_{K}~:=~\{(j\circ i)^{\ast}:G/L^{a}\to G/H~:\ a\in G~
\text{and}~j:K\hra L^{a}~~\text{is an inclusion}\}~\subseteq~\tEH{m}~.
$$
\noindent Note that each object in \w{\cV^{L}_{K}} is determined by
the conjugate \w{L^{a}} (containing $K$).

In particular, if \w[,]{K=H} we let
\w[.]{\cP^{L}_{H}:=\cV^{L}_{H}\cap\sk\tEH{m}} Moreover, given
\w[,]{\cP^{L}_{H}} we can generate all of \w{\cV^{L}_{H}} by
conjugating each \w{j=j\circ\Id_{H}} by all possible elements
\w[,]{a\in N_{G}H} so
\w[.]{\cV^{L}_{H}=\coprod_{j^{\ast}\in\cP^{L}_{H}}~\lra{j^{\ast}}}
Since \w[,]{\cV^{L}_{K}\subseteq\cV^{L}_{H}} this yields a partition:
\begin{myeq}[\label{eqpartition}]
\cV^{L}_{K}~=~\coprod_{(j\circ i)^{\ast}:G/L^{a}\to G/H \in\cP^{L}_{H}}~~
\lra{(j\circ i)^{\ast}}~\cap~\cV^{L}_{K}
\end{myeq}
\noindent for any \w[.]{H\leq K\leq G}

Now consider a cell \w{\uD^{n}_{G/L}} in \w{\TOG} (\S
\ref{scellular}). By definition, for each \w[,]{G/K\in\OG\op}
$$
\uD^{n}_{G/L}(G/K)=\coprod_{\psi\op:G/L\to
  G/K~\text{in}~\OG\op}~\bD{n}_{\psi}
$$
\noindent (see \wref[).]{eqncell} However, every
map \w{\psi\op:G/L\to G/K} is determined by choosing a conjugate
\w{L^{a}} of $L$ containing $K$, with inclusion \w[,]{j:K\hra L^{a}}
and then \w{\psi\op=j^{\ast}\circ\co{L}{a}} (see \S \ref{socat}).
Since \w{\co{L}{a}} is an isomorphism, the possible choices of such
isomorphisms \w{G/L\to G/L^{a}} (for fixed \w[)]{L^{a}\in\cV^{L}_{K}}
is in one-to-one correspondence with \w[.]{\Aut(G/L)=\WL}
Thus we have a set bijection between \w{\Hom_{\OG\op}(G/L,G/K)} and
\w[,]{\cV^{L}_{K}\times\WL} so:
%
%\begin{myeq}[\label{eqcelllk}]
$$
\uD^{n}_{G/L}(G/K)~=~\coprod_{\cV^{L}_{K}}~~\coprod_{\WL}~~\bD{n}
~=~\coprod_{(j\circ i)^{\ast}\in\cP^{L}_{H}}~
\coprod_{\lra{(j\circ i)^{\ast}}~\cap~\cV^{L}_{K}}~
\coprod_{\WL/\bWHL}~\coprod_{\bWHL}~~\bD{n}~,
$$
%\end{myeq}
%
\noindent where the second equality follows from \wref{eqpartition}
and Definition \ref{dsubweyl}.

On the other hand, let \w{\hD^{n}_{\vartheta\op}} be an $n$-cell in
\w[,]{\TE{m}} for some \w{\vartheta\op:G/L\to G/H} in
\w[.]{\tEH{m}} Given any \w{\zeta\op: G/K\to G/H} in \w[,]{\tEH{m}}
we have:
$$
\hD^{n}_{\vartheta\op}(\zeta\op)~=~\coprod_{\tilde{\psi}\op:\vartheta\op\to
\zeta\op~\text{in}~\tEH{m}}~\bD{n}_{\tilde{\psi}}~.
$$
\noindent We may assume for simplicity that \w{\vartheta\op} and
\w{\zeta\op} are in the skeleton \wh that is, they are induced by
inclusions: \w{\vartheta\op=\ell^{\ast}:G/L\to G/H} and
\w[.]{\zeta\op=i^{\ast}: G/K\to G/H}
Because \w{\tU{m}:\tEH{m}\to\OG\op} is faithful, the
possible maps \w{\tilde{\psi}\op} are determined by
\w{\psi\op:=\tU{m}(\tilde{\psi}\op):G/L\to G/K} in \w[,]{\OG\op}
that is, by the choice of \w{a\in N_{G}H} and
\w{(j\circ i)^{\ast}\in\lra{\ell^{\ast}}\cap\cV_{K}^{L}} (for
\w[),]{j:K\hra L^{a}} and an isomorphism \w{G/L\to G/L^{a}} in
\w[.]{\tEH{m}} Therefore:
%
%\begin{myeq}[\label{eqcellehmlk}]
$$
\hD^{n}_{\ell^{\ast}}(i^{\ast})~=~\coprod_{\lra{\ell^{\ast}}\cap\cV_{K}^{L}}~
\coprod_{\bWHL}~\bD{n}~.
$$
%\end{myeq}

We conclude that for any cell \w{\uD^{n}_{G/L}} in \w[,]{\TOG}
its restriction \w{\uD^{n}_{G/L}\circ\tU{m}} to
\w{\TE{m}} is itself a coproduct of cells, viz.:
$$
\uD^{n}_{G/L}\circ\tU{m}~=~
\coprod_{(j\circ i)^{\ast}\in\cP^{L}_{H}}~\coprod_{\WL/\bWHL}~~
\hD^{n}_{(j\circ i)^{\ast}}~.
$$
\noindent The cellularity of \w{\uX\circ\tU{m}} now follows by
induction from \wref{eqattach} and Fact \ref{funique}.
\end{proof}

\begin{remark}\label{rwhsystem}
Note that for \w{L>H} we have \w[,]{N_{G}H\cap N_{G}L\subseteq N_{G}(N_{G}H\cap L)}
so we have an inclusion
%
%\begin{myeq}[\label{eqnormal}]
$$
(N_{G}H\cap N_{G}L)/(N_{G}H\cap L)~=~\bWHL~\hra~W_{N_{G}H\cap L}~.
$$
%\end{myeq}
%
\noindent Therefore, \w{\bWHL} acts on \w{\uM(G/(N_{G}H\cap L))} for
any \ww{\OG\op}-diagram $\uM$.

Furthermore, there is a one-to-one correspondence between the lattice
of groups $L$ with \w{H\leq L\leq N_{G}H} and the lattice of subgroups
\w{\bar{L}:=L/H} of \w[,]{\WH=N_{G}H/H} which induces an inclusion of
categories \w{I:\OWH\op\hra\OG\op} Therefore, given a diagram
\w[,]{\uM:\OG\op\to\Abgp} we can define a coefficient system
\w[,]{\hM=\uM\circ I:\OWH\op\to\Abgp} with \w[.]{\hM(\WH/\bar{L}):=\uM(G/L)}

Moreover, if \w{\bY:=\bK(\uM,n)^{H}} as above, then
\w{\bY^{L}\simeq \bK(\uM(G/L),n)} for any \w[,]{h\leq L\leq N_{G}H} so
$\bY$ is a \ww{\WH}-space of type \w{\bK(\hM,n)} (see \S
\ref{sbredon}), and thus
$$
\HB{i}{\WH}{\bX^{H}}{\hM}~\cong~\pi_{n-i}\map_{\WH}(\bX^{H},\bY)
$$
\noindent for any $G$-space $\bX$ by \cite[II, \S 1]{MayEH}.
\end{remark}

\begin{thm}\label{tssfps}
For any finite group $G$, $G$-CW complex $\bX$, and coefficient system
\w[,]{\uM:\OG\op\to\Abgp} for each subgroup \w{H\leq G} there is a
first quadrant spectral sequence with:
$$
E_{1}^{m,i}~=~\bigoplus_{G/L\in\sk\tSH{m+i}}
\tH{i}{\bWHL}{\bE\WL\times_{\bWHL}\XLL}{M_{N_{G}H\cap L}}~~
\Longrightarrow~~\HB{i}{\WH}{\bX^{H}}{\hM}~.
$$
\end{thm}

\begin{proof}
For \w{\hX,\hY:\tEH{k}\to\Top} of Definition \ref{dfiltdiag}, as in \S
\ref{sstrata}, we can restrict $\hX$ and $\hY$ along
\wref{eqfilterd} to obtain \w[.]{\tau_{m}\hX,\tau_{m}\hY:\tEH{m}\to\Top}
By Corollary \ref{ccofb}, when $\bX$ is a $G$-CW complex, $\uX$ and
thus \w{\tau_{m}\uX} are cellular, and thus cofibrant in \w{\TOG} and
\w[,]{\Top^{\F_{m}}} respectively. Applying Proposition
\ref{pcellular} we deduce that \w{\hX:=\uX\circ\tU{k}} and each
\w{\tau_{m}\hX:=\uX\circ\tU{m}} \wb{m\leq k} is cofibrant,
too.

As in \S \ref{sstrata} the truncation functor \w{\TE{m+1}\to\TE{m}}
has a left adjoint \w[.]{\zeta_{m}:\TE{m}\to\TE{m+1}} As in the proof
of Proposition \ref{pcofib}, one can show that the counit
\w{\varepsilon_{m-1}:\zeta_{m-1}\tau_{m-1}\hX\to\tau_{m}\hX} is a
cellular inclusion, and thus a cofibration, in \w[.]{\TE{m}} Therefore,
applying \w{\map_{\TE{m}}(-,\tau_{m}\hY)} to \w{\varepsilon_{m-1}}
yields a fibration
$$
\varepsilon_{m-1}^{\ast}:\map_{\TE{m}}(\tau_{m}\hX,\tau_{m}\hY)\to
\map_{\TE{m}}(\zeta_{m-1}\tau_{m-1}\hX,\tau_{m}\hY)=
\map_{\TE{m-1}}(\tau_{m-1}\hX,\tau_{m-1}\hY)
$$
\noindent for each \w[.]{m\leq k} We therefore obtain a (finite) tower
of fibrations
$$
\map(\tau_{k}\hX,\tau_{k}\hY)~\xra{p_{k}}~\dotsc~
\map(\tau_{m}\hX,\,\tau_{m}\hY)~\xra{p_{m}}~
\map(\tau_{m-1}\hX,\,\tau_{m-1}\hY)~\to~\dotsc~
$$
\noindent (mapping spaces in \w[),]{\TE{m}} as in \wref[,]{eqtower}
with successive fibers \w[.]{\hat{F}_{m}(\hX,\hY)}

The homotopy spectral sequence for this tower converges to the
homotopy groups of \w[,]{\map(\tau_{k}\hX,\tau_{k}\hY)=\map(\hX,\hY)}
which are the required Bredon cohomology groups by Lemma
\ref{lmapspace} and \wref[.]{eqelmendorf} The \ww{E_{1}}-term of the
spectral sequence can be described in terms of the homotopy groups of
the fibers \w[,]{\hat{F}_{m}(\hX,\hY)} which we identify as follows:

By definition,
$$
\hat{F}_{m}(\hX,\hY)=\{\hf\in\map_{\tEH{m}}(\hX,\hY)~:\ \tau_{m-1}\hf=0\}~.
$$
\noindent and since \w{\map_{\tEH{m}}(\hX,\hY)} is a subspace of
\w[,]{\prod_{\vartheta\op\in\tEH{m}}~\map_{\Top}(\hX(\vartheta\op),\,
\hY(\vartheta\op))}
we see that \w{\hat{F}_{m}(\hX,\hY)} is a subspace of
\w[,]{\prod_{\vartheta\op\in\tSH{m}}~\map_{\Top}(\hX(\vartheta\op),\,
\hY(\vartheta\op))} and in fact, it suffices to consider the factors
indexed by objects in a skeleton of \w[.]{\tSH{m}} Note that any map
of diagrams must be equivariant with respect to the action of the
automorphisms of such objects, so:
$$
\hat{F}_{m}(\hX,\hY)~\subseteq~
\prod_{G/L\in\sk\tSH{m}}~~\map_{\Aut_{\tSH{m}}(G/L)}(\hX(G/L),~\hY(G/L))~.
$$

In our case \w[,]{\hX(G/L)=\bX^{L}} \w[,]{\hY(G/L)=\bK(M_{N_{G}H\cap L},n)} and
\w[,]{\Aut_{\tSH{m}}(G/L)=\bWHL\subseteq\WL} so
$$
\hat{F}_{m}(\hX,\hY)\subseteq
\prod_{G/L\in\sk\tSH{m}}~\map_{\bWHL}(\bX^L,\,\bK(M_{N_{G}H\cap L},n))~.
$$

However, in order for a collection of \ww{\bWHL}-maps
\w{f_{L}:\bX^L\to\bK(M_{N_{G}H\cap L},n)} \wb{G/L\in\sk\tSH{m}} to yield a map
of diagrams \w{\hf:\hX\to\hY} in \w[,]{\tEH{m}} we need the
restrictions to \w{\hX(G/K)=\bX^{K}} to vanish for any
\w[.]{K\geq L} Since \w{\bX_{L}=\bigcup_{K>L}~\bX^{K}} is a
\ww{\bWHL}-invariant subspace of \w[,]{\bX^{L}} this is equivalent to
each map \w{f_{L}} inducing a (pointed) \ww{\bWHL}-map
\w{\XLL=\bX^{L}/\bX_{L}\to\bK(M_{N_{G}H\cap L},n))}

Therefore,
$$
\hat{F}_{m}(\hX,\hY)~\cong~\prod_{G/L\in\sk\tSH{m}}~
\map_{\WLH}(\XLL,\,\bK(M_{H},n))~.
$$
\noindent where \w{\WLH} acts on \w{\bK(M_{N_{G}H\cap L},n)}
via the quotient map \w{\pi:\WLH\to\bWHL} as a subgroup of
\w[,]{W_{N_{G}H\cap L}} and on \w{\XLL} via $\pi$.

From Fact \ref{ffree} and Lemma \ref{lfreeptact} we deduce:
$$
\pi_{i}\hat{F}_{m}(\hX,\hY)~\cong~\bigoplus_{G/L\in\sk\tSH{m}}\
\tH{n-i}{\WLH}{\bE\WL\times_{\bWHL}\XLL}{M_{N_{G}H\cap L}}~.
$$
\end{proof}

\begin{remark}\label{rtwoss}
Given a $G$-CW complex $\bX$ and a coefficient system $\uM$, we
have two spectral sequences for computing the Bredon cohomology groups
\w[,]{\HB{\ast}{\WH}{\bX^{H}}{\hM}} given by Theorems \ref{tss} and
\ref{tssfps}, respectively, both starting with reduced
cohomology groups with local coefficients. However, they are not identical, since the modified fixed point sets appearing in the \ww{E_{1}}-term of
the first spectral sequence are restricted to
\w{(\bX^{H})\sp{\bar{L}}\sb{\bar{L}}} for \w[,]{\bar{L}\leq\WH} while
in the second spectral sequence we allow all \w{\XLL} for
\w[.]{H\leq L\leq G} Note, however, that if \w[,]{N_{G}H\leq L} then
\w{\bWHL} is trivial, so by \wref{eqsseseq} the reduced
cohomology is just ordinary (non-equivariant) cohomology.
\end{remark}

\begin{remark}\label{rhomology}
For finite $G$ and any $G$-CW complex $\bX$, Bredon also defined its cellular equivariant homology groups, and this definition was extended to arbitrary $G$-spaces by Br\"{o}cker in \cite{BrockS}, and to arbitrary topological groups by
Willson in \cite{WillsE}. However, these constructions are in terms of appropriate chain complexes, rather than mapping spaces of diagrams as in \wref[,]{eqbredon}
so our methods do not yield analogous spectral sequences in Bredon homology.
\end{remark}

\end{document}